\documentclass{amsart}
\usepackage{amscd,amssymb,amsopn,amsmath,amsthm,graphics,amsfonts,accents,enumerate,verbatim,calc}
\usepackage[dvips]{graphicx}
\usepackage[colorlinks=true,linkcolor=red,citecolor=blue]{hyperref}
\usepackage[all]{xy}
\usepackage{cite}
\usepackage{tikz}
\usepackage{tikz-cd}
\usepackage{mathrsfs}

\calclayout

\newcommand{\lrt}{\longrightarrow}

\newcommand{\st}{\stackrel}

\newcommand{\la}{\lambda}
\newcommand{\La}{\Lambda}

\newcommand{\Z}{\mathbb{Z}}

\newcommand{\SA}{\mathscr{A}}
\newcommand{\SB}{\mathscr{B}}

\newcommand{\SK}{\mathscr{K}}

\newcommand{\SX}{\mathscr{X}}
\newcommand{\SY}{\mathscr{Y}}

\newcommand{\CA}{\mathcal{A} }

\newcommand{\CC}{\mathcal{C} }

\newcommand{\CE}{\mathcal{E}}
\newcommand{\CF}{\mathcal{F} }

\newcommand{\CH}{\mathcal{H}}
\newcommand{\CI}{\mathcal{I} }
\newcommand{\CJ}{\mathcal{J} }

\newcommand{\CM}{\mathcal{M} }
\newcommand{\CN}{\mathcal{N} }
\newcommand{\CP}{\mathcal{P} }

\newcommand{\CR}{\mathcal{R} }
\newcommand{\CS}{\mathcal{S} }
\newcommand{\CT}{\mathcal{T} }
\newcommand{\CU}{\mathcal{U}}
\newcommand{\CV}{\mathcal{V}}
\newcommand{\CW}{\mathcal{W}}
\newcommand{\CX}{\mathcal{X} }

\newcommand{\Mod}{{\rm{Mod\mbox{-}}}}

\newcommand{\mmod}{{\rm{{mod\mbox{-}}}}}
\newcommand{\mmodd}{{\rm{mod}}_F\mbox{-}}

\newcommand{\ind}{\rm{{ind}\mbox{-}}}

\newcommand{\op}{{\rm{op}}}

\newcommand{\add}{{\rm{add}\mbox{-}}}

\newcommand{\Coker}{{\rm{Coker}}}

\newcommand{\rad}{{\rm{rad}}}
\newcommand{\Ind}{{\rm{Ind}}}

\newcommand{\Hom}{{\rm{Hom}}}
\newcommand{\Ext}{{\rm{Ext}}}

\newcommand{\End}{{\rm{End}}}

\newcommand{\Tr}{{\rm{Tr}}}

\theoremstyle{plain}
\newtheorem{theorem}{Theorem}[section]
\newtheorem{corollary}[theorem]{Corollary}
\newtheorem{lemma}[theorem]{Lemma}

\newtheorem{proposition}[theorem]{Proposition}

\newtheorem{notation}[theorem]{Notation}

\theoremstyle{definition}
\newtheorem{definition}[theorem]{Definition}

\newtheorem{remark}[theorem]{Remark}
\newtheorem{setup}[theorem]{Setup}
\newtheorem{s}[theorem]{}

\theoremstyle{plain}

\theoremstyle{definition}

\numberwithin{equation}{section}

\begin{document}

\title[Relative Higher Homological Algebra]{Relative Higher Homology and representation theory}

\author[Rasool Hafezi, Javad Asadollahi and Yi Zhang]{Rasool Hafezi, Javad asadollahi and Yi Zhang}

\address{School of Mathematics and Statistics, Nanjing University of Information Science \& Technology, Nanjing, Jiangsu 210044, P.\,R. China}
\email{hafezi@nuist.edu.cn}

\address{Department of Pure Mathematics, Faculty of Mathematics and Statistics, University of Isfahan, P.O.Box: 81746-73441, Isfahan, Iran}
\email{asadollahi@ipm.ir, asadollahi@sci.ui.ac.ir}

\address{School of Mathematics and Statistics, Nanjing University of Information Science \& Technology, Nanjing, Jiangsu 210044, P.\,R. China}
\email{zhangy2016@nuist.edu.cn}



\begin{abstract}
Higher homological algebra, basically done in the framework of an $n$-cluster tilting subcategory $\CM$ of an abelian category $\CA$, has been the topic of several recent researches. In this paper, we study a relative version, in the sense of Auslander-Solberg, of the higher homological algebra. To this end, we consider an additive sub-bifunctor $F$ of $\Ext^n_{\CM}( - , - )$ as the basis of our relative theory. This, in turn, specifies a collection of $n$-exact sequences in $\CM$, which allows us to delve into the relative higher homological algebra. Our results include a proof of the relative $n$-Auslander-Reiten duality formula, as well as an exploration of relative Grothendieck groups, among other results. As an application, we provide necessary and sufficient conditions for $\CM$ to be of finite type. 
\end{abstract}

\maketitle

\tableofcontents

\section{Introduction and preliminaries}
\s The importance of relative homological algebra in the representation theory of artin algebras is highlighted in the three influential papers by M. Auslander and {\O}. Solberg \cite{ASo}, \cite{ASo2} and \cite{ASo3}. One way to approach this theory is by examining short exact sequences that fall within a sub-bifunctor $F$ of $\Ext^1$. Let $\La$ be an artin algebra. In \cite{ASo} the authors established nice connections between certain additive sub-bifunctors of the extension functor $\Ext^1_{\La}( - , - )$ and homologically finite subcategories of $\mmod\La.$ For instance, they showed that there is a bijection between the collection of additive sub-bifunctors $F$ of $\Ext^1_{\La}( - , - )$ with enough projectives and the collection of contravariantly finite subcategories of $\mmod\La$ containing $\CP(\La)$, the subcategory of projective modules in $\mmod \La$. Dually, there is a bijection between the collection of additive subfunctors $F$ of $\Ext^1_{\La}( - , - )$ with enough injectives and the collection of all covariantly finite subcategories of $\mmod \La$ containing $\CI(\La)$, where $\CI(\La)$ denotes the subcategory of injective modules in $\mmod \La$.

Our initial motivation for this work is to elaborate on such bijections, and other similar results, in the new setting of higher homological algebra. Let us be more precise.

\s Higher homological algebra, known also as $n$-homological algebra, is a vast generalization of homological algebra, having sequences of length $n+2$ playing the role of the short exact sequences in ($1$-)abelian categories. It appeared in a series of papers by Iyama \cite{I1, I2, I3} and then axiomatized and studied extensively by Jasso \cite{Ja}.

To recall the notion, let $\CM$ be an additive category and $n \geq 1$ be a fixed integer. Let $f^0 : M^0 \lrt M^1$  be a morphism in $\CM$. An $n$-cokernel of $f^0$  is a sequence
\[M^1 \st{ f^1}{\lrt } M^2 \lrt \cdots \lrt M^n\st{ f^n}{ \lrt } M^{n+1}\]
of objects and morphisms in $\CM$ such that for every $M \in \CM$, the induced sequence
\[0 \lrt \CM(M^{n+1}, M) \lrt \CM(M^n, M) \lrt \cdots \lrt \CM(M^2, M) \lrt \CM(M^1, M) \lrt \CM(M^0, M) \]
of abelian groups is exact. We usually denote an $n$-cokernel of $f^0$ by $(f^1, f^2, \ldots, f^n)$. Dually we define the notion of an $n$-kernel of a morphism.

An $n$-exact sequence in $\CM$ is a sequence
\[0 \lrt M^0 \st{f^0}{\lrt} M^1 \st{ f^1}{\lrt } M^2 \lrt \cdots \lrt M^n\st{ f^n}{ \lrt } M^{n+1} \lrt 0\]
such that $(f^0, f^1, \ldots, f^{n-1})$ is an $n$-kernel of $f^n$ and $(f^1, f^2, \ldots, f^n)$ is an $n$-cokernel of $f^0$ \cite[Definitions 2.2, 2.4]{Ja}.

Let $\CM$ be an additive category. We say that $\CM$ is an $n$-abelian category \cite[Definition 3.1]{Ja} if it is idempotent complete, each morphism in it admits an $n$-cokernel and an $n$-kernel and every monomorphism $f^0: M^0 {\lrt } M^1$, resp. every epimorphism $f^n: M^n { \lrt } M^{n+1}$, can be completed to an $n$-exact sequence
\[0 \lrt M^0 \st{f^0}{\lrt} M^1 \st{ f^1}{\lrt } M^2 \lrt \cdots \lrt M^n\st{ f^n}{ \lrt } M^{n+1} \lrt 0.\]

\s The best-known examples of $n$-abelian categories are $n$-cluster tilting subcategories of abelian categories. Cluster tilting subcategories are introduced by Iyama \cite[Definition 2.2]{I1} in studying a higher version of Auslander's correspondence, see also \cite{I2} and \cite{I3}. 

Let $\SA$ be an abelian category. Recall that a subcategory $\SB$ of $\SA$ is called a covariantly finite subcategory if for every $A \in \SA$ there exists an object $B\in\SB$ and a morphism $f : A \lrt B$ such that for all $B' \in\SB$ the sequence of abelian groups $\SA(B, B') \lrt \SA(A, B') \lrt 0$ is exact. Such a morphism $f$ is called a left $\SB$-approximation of $A$. Contravariantly finite subcategories and right $\SB$-approximations are defined dually. A functorially finite subcategory of $\SA$ is a subcategory that is both a covariantly finite and a contravariantly finite subcategory \cite[page 113]{AR}. 

By definition, an additive subcategory $\CM$ of $\SA$ is called an $n$-cluster tilting subcategory if it is a functorially finite and generating-cogenerating subcategory of $\SA$ and satisfies the equalities $\CM^{\perp_n}=\CM= {}^{\perp_n}\CM$, where
\[ \ \ \ \ \ \ \CM^{\perp_n}:= \{X \in \SA \mid \Ext^i_{\SA}(\CM, X)=0, \ \text{ for all} \ 0 < i<  n \},  \ {\rm and}\]
\[{}^{\perp_n}\CM:= \{X \in \SA \mid \Ext^i_{\SA}(X, \CM)=0, \ \text{ for all} \ 0 < i <  n \}.\]
 It is known that $\CM$ has a structure of an $n$-abelian category  \cite[Theorem 3.16]{Ja}.

\s\label{ASadeghi} After introducing the concept of $n$-abelian categories, there have been several attempts to study the various notions of classical homological algebra in these new settings, see for instance, \cite{IJ}, \cite{Jor}, \cite{HJV}, \cite{JJ}, \cite{AMS} and \cite{EN}. In particular, higher $\Ext$ groups have been studied in \cite{Ja} and \cite{Lu}. 

Let $\CM$ be an $n$-cluster tilting subcategory of $\mmod\La$. Let $N$ and $L$ be in $\CM$. 
An $n$-exact sequence
\begin{equation}
\xi \colon 0 \lrt L\lrt M^1\lrt \cdots \lrt M^n\lrt N \lrt 0 \notag
\end{equation}
is called an $n$-extension of $N$ by $L$. Two $n$-extensions $\xi$ and $\xi'$ of $N$ by $L$ are said to be Yoneda equivalent if there exists a chain of morphisms of $n$-exact sequences
\begin{equation}
\xi=\xi_0,  \xi_1, \ldots ,\xi_{l-1}, \xi_l=\xi' \notag
\end{equation}
such that for every $i\in \{0,\ldots,l-1\}$, we have either a chain map $\xi_i\lrt \xi_{i+1}$ or a chain map $\xi_{i+1}\lrt \xi_{i}$ starting with $1_A$ and ending with $1_L$.  By \cite[Proposition 4.10]{Ja}, this is an equivalence relation. The set of all Yoneda equivalence classes of $n$-extensions of $N$ by $L$ with baer sum is an abelian group \cite[Remark 6.44]{FS}, that will be denoted by $\Ext^n_{\CM}(N, L)$. In fact, $\Ext^n_{\CM}( - ,  - )$ is an additive bifunctor on $\CM^{\op}\times\CM$, which means that, for every object $X$ in $\CM$, the functors $\Ext^n_{\CM}(X ,- ) : \CM \lrt \CA\rm{b}$ and $\Ext^n_{\CM}( - , X) : \CM^{\rm op} \lrt \CA\rm{b}$ are additive \cite[Section 1]{ASo}, where $\CA {\rm b}$ denotes the category of abelian groups. 

From now on, for simplicity, we refer to additive sub-bifunctors of $\Ext^n_{\CM}$ as subfunctors.
It is known \cite[2.4]{ASa} that an additive subfunctor $F$ of $\Ext^n_{\CM}( - , - ): \CM^{\rm op} \times \CM \lrt \CA{\rm b}$ determines a collection of $n$-exact sequences in $\CM$ which is closed under isomorphisms, direct sums and $n$-pullbacks and $n$-pushouts along any other morphism in $\CM$. Conversely, a collection of $n$-exact sequences satisfying these properties gives rise to an additive subfunctor of  $\Ext^n_{\CM}( - , - )$. The $n$-exact sequences in $F$ will be called $n$-$F$-exact sequences.

\s In this paper, we examine the relative homological algebra within an $n$-cluster tilting subcategory $\CM$ of $\mmod\La$, where $\La$ is an artin algebra and $n$ is a fixed positive integer. The paper is structured as follows. In Section \ref{Section 2}, we consider a subfunctor $F$ of $\Ext^n_{\CM}( - , - )$, and explore some homological concepts in the collection of $n$-exact sequences that belong to $F$. In particular, we establish a correspondence between the collection of subfunctors of $\Ext^n_{\CM}( - , - )$ with enough projectives and the collection of contravariantly finite subcategories of $\CM$ that contain projective modules. 

Section \ref{Section 3} is devoted to the study of the monomorphism category within an $n$-cluster tilting subcategory of $\mmod\La$ with respect to a subfunctor $F$. Our main result in this section, i.e. Theorems \ref{equiv-n-first} and \ref{duality-n-exat2}, provide equivalences/dualities that will be used in the Section \ref{Section 4}, where we prove a relative $n$-Auslander-Reiten duality formula, thus provide a generalization of Iyama’s higher Auslander–Reiten
duality formula given as Theorem 2.3.1 in \cite{I1}, see Theorem \ref{Theroem-relative-tau}. 

Grothendieck groups in higher homological algebra settings have been studied in \cite{R} and  \cite{DN}. In Section \ref{Section 5}, we present a relative version of the higher Grothendieck group. As a result, we not only generalize the main result of \cite{DN}, namely Theorem 3.11, but also remove the strong assumption required in their work, see Theorem \ref{eqofgro}. The last section of the paper provides the necessary and sufficient conditions for $\CM$ to be of finite type, in terms of relative higher homological algebra. Our results in this section generalize those of \cite{Au3} and \cite{E}. 

\s We end this introductory section by recalling notions and definitions we need. 
Throughout $\La$ is an artin algebra over a commutative artinian ring $R$. Let
\[\Tr : \underline{{\rm mod}}\mbox{-}\La   \leftrightarrow  \underline{{\rm mod}}\mbox{-}\La^{\op}, \ \ \ \Omega: \underline{{\rm mod}}\mbox{-}\La \lrt \underline{{\rm mod}}\mbox{-}\La \ \ \ {\rm and}
 \ \ \ \Omega^{-{}}: \overline{{\rm mod}}\mbox{-}\La \lrt \overline{{\rm mod}}\mbox{-}\La\]
denote, respectively, the Auslander-Bridger transpose duality, the syzygy functor and the cosyzygy functor \cite{AB}, where for an additive category $\SA$, $\underline{\SA}$ and $\overline{\SA}$ denote the stable and costable categories of $\SA$, respectively \cite{ARS}.

For an integer $n>1$, the $n$-Auslander-Reiten translations $\tau_n$ and $\tau_n^{-}$ are introduced and studied in \cite{I1} as natural generalizations of the classical Auslander-Reiten translations $\tau$ and $\tau^{-}$. The $n$-Auslander-Reiten translations are defined by
\[ \ \ \ \ \ \tau_n := D\Tr \Omega^{n-1} : \underline{{\rm mod}}\mbox{-}\La \lrt \overline{{\rm mod}}\mbox{-}\La, \ \ \ {\rm and}\]
\[\tau^{-}_n := \Tr D \Omega^{-(n-1)} : \overline{{\rm mod}}\mbox{-}\La \lrt \underline{{\rm mod}}\mbox{-}\La.\]
Here $D$ denotes the duality $D( - ) := \Hom_R( - , E),$ where $E$ is the injective envelope of the $R$-module $R/{\rad}R$.
As it is clear from the definition, $\tau_n = \tau\Omega^{n-1}$ and $\tau^{-}_n = \tau^{-}\Omega^{-(n-1)}$. By Proposition 1.2 of \cite{I3}, for any $n$-cluster tilting subcategory $\CM$ of $\mmod\La$, the functor $\tau_n$ provides a bijection between the set of iso-classes of indecomposable non-projective objects in $\CM$ and the set of iso-classes of indecomposable non-injective objects in $\CM$.

\s \label{almost split} Let $\CM$ be an $n$-cluster tilting subcategory of $\mmod\La$. By \cite[Definition 3.1]{I1}, an $n$-almost split sequence is an exact sequence
\[0 \lrt M^0 \st{f^0}\lrt M^1 \st{f^1}\lrt M^2 \lrt \cdots \lrt M^{n} \st{f^{n}}{\lrt} M^{n+1} \lrt 0 \]
in $\CM$ such that for all $i $ in $\{0,1, \ldots, n\}$, $f^i \in \CJ_{\La}$, where $\CJ_{\La}$ is the Jacobson radical of $\mmod\La$, and the sequence
\[0 \lrt \Hom_{\La}(M, M^0) \lrt \Hom_{\La}(M, M^1) \lrt \cdots \lrt \Hom_{\La}(M, M^{n}) \lrt \CJ_{\La}(M, M^{n+1}) \lrt 0 \]
is exact for every $M$ in $\CM$. 

By Theorem 3.3.1 of \cite{I1}, for every non-projective, resp. non-injective, indecomposable module $X$, resp. $Y$, in $\CM$, there exist $n$-almost split sequences
\[ \ \ \ \ \ \ 0 \lrt \tau_nX \lrt M^1 \lrt \cdots \lrt M^n \lrt X \lrt 0, \ {\rm and}\]
\[\ 0 \lrt Y \lrt M'^1 \lrt  \cdots \lrt M'^n \lrt \tau_n^{-1}Y  \lrt 0, \]
in $\CM$, respectively. For a recent study of  $n$-almost split sequences, in the framework of $n$-exangulated categories, see \cite{HHZZ}.

\s Let $\CC$ be an essentially small additive category. By definition, a (right) $\CC$-module is a contravariant additive functor $F: \CC \lrt \mathcal{A}\rm{b}$. The collection of $\CC$-modules and natural transformations between them form an abelian category denoted by $\Mod \CC$. A $\CC$-module $F$ is called finitely presented if there exists an exact sequence
\[\CC( - , C) \lrt \CC( - , C') \lrt F \lrt 0,\]
with $C$ and $C'$ in $\CC$. All finitely presented $\CC$-modules form a full subcategory of $\Mod\CC$, denoted by $\mmod \CC.$ We often write $\CC \mbox{-} \text{mod}$ instead of $\mmod \CC^{\rm op}$. It is known that if $\CC$ admits weak kernels, then $\mmod \CC$ is an abelian category, see\cite[\S III, Section 2]{Au}. In particular, if $\CC$ is a contravariantly finite subcategory of an abelian category $\CA$, then it admits weak kernels and hence $\mmod\CC$ is an abelian category.
Since an $n$-cluster tilting subcategory $\CM$ of $\mmod\La$ is functorially finite, $\mmod\CM$ and $\CM\mbox{-}\text{mod}$ are abelian categories.

\begin{setup}
Throughout the paper, $\CM$ denotes an $n$-cluster tilting subcategory of $\mmod\La$, where $n$ is a fixed positive integer. We also assume that $F$ is a fixed additive subfunctor of $\Ext^n_{\CM}( - , - )$. 
\end{setup}

\section{Subfunctors of $\Ext^n_{\CM}( - , - )$}\label{Section 2}
In this section, we study some basic notions in the context of higher homological algebra. Let $F$ be a subfunctor of $\Ext^n_{\CM}( - , - )$. Elements of $F$, i.e. $n$-exact sequences in $F$, will be called $n$-$F$-exact sequences.

\begin{definition}
A module $P$ in $\CM$ is called $n$-$F$-projective if for every $n$-$F$-exact sequence
\begin{equation}
\xi \colon 0 \lrt M' \lrt M^1 \lrt \cdots \lrt M^n \lrt M \lrt 0 \notag
\end{equation}
the induced sequence
\begin{equation}
0 \lrt \CM(P, M') \lrt \CM(P, M^1)\lrt \cdots \lrt \CM(P, M^n) \lrt \CM(P, M) \lrt 0 \notag
\end{equation}
is exact. We say that $F$ has enough $n$-$F$-projectives, for simplicity, enough projectives, if for every $M \in \CM$, there exists an $n$-$F$-exact sequence
\begin{equation}
0 \lrt M^1\lrt M^2\lrt \cdots \lrt  M^n\lrt P\lrt M \lrt 0, \notag
\end{equation}
with $P$ an $n$-$F$-projective module. Similarly, we define the notions of $n$-$F$-injective modules in $\CM$ and a subfunctor having enough injectives. 
\end{definition}

The subcategories of $\CM$ consisting of $n$-$F$-projective, resp. $n$-$F$-injective, modules are denoted by $\CP(F)$, resp. $\CI(F)$. If $F=\Ext^n_{\La}(-, -)$, then $\CP(F)=\CP(\La)$ and $\CI(F) = \CI(\La)$.

\begin{notation}\label{notation}
Since, throughout the paper, $n$ is a fixed positive integer, for simplicity $n$-$F$-exact sequences, $n$-$F$-projective modules and $n$-$F$-injective modules will be referred to as $F$-exact sequences, $F$-projective modules and $F$-injective modules, respectively.
\end{notation}

\begin{lemma}\label{induced-exact-sequence}
Let $\xi: 0 \lrt M' \lrt M^1\lrt \cdots \lrt M^n\lrt M \lrt 0$ be an $F$-exact sequence. For all $ X\in \CM$, there exist exact sequences
\[0 \lrt \CM(X, M')\lrt \CM(X,M^1)\lrt\cdots\lrt \CM(X,M^n)\lrt\CM(X, M) \st{\delta}\lrt F(X, M')  \]
and
\[0 \lrt \CM(M, X)\lrt \CM(M^n, X)\lrt\cdots\lrt \CM(M^1, X)\lrt \CM(M',X) \st{\delta}\lrt F(M, X) \]
of abelian groups.
\end{lemma}

\begin{proof}
We only prove the existence of the first exact sequence; the proof for the second is omitted as it is dual. By \cite[Lemma 3.5]{I3}, we have the following exact sequence
\begin{align}
0&\lrt \CM(X, M')\lrt \CM(X,M^1)\lrt\cdots\lrt \CM(X,M^n)\lrt \CM(X, M) \st{\delta}\lrt \Ext^n_{\CM}(X, M') \notag
\end{align}
Let $g:X \lrt M$ be a morphism in $\CM(X, M)$. The connecting map $\delta$ maps $g$ to $\xi g$, by taking $n$-pullback of $\xi$ along $g$. Since $\xi$ is $F$-exact, $\xi g$ is an element of $F(X, M')$. This completes the proof.
\end{proof}

An $n$-exact sequence
\[\xi \colon 0 \lrt M^0 \st{f^0}\lrt M^1 \st{f^1}\lrt M^2 \lrt \cdots \lrt M^{n} \st{f^n}{\lrt} M^{n+1} \lrt 0 \]
is called split, or contractible, if the identity morphism on $\xi$ is null-homotopic. By \cite[Proposition 2.6]{Ja}, $\xi$ is split if and only if $f^0$ is a split monomorphism if and only if $f^n$ is a split epimorphism.

\begin{corollary}\label{vanishing-proj-inj}
The following statements hold.
\begin{itemize}
    \item [$(i)$] $P$ is $F$-projective if and only if $F(P, M)=0$, for all modules $M$ in $\CM.$
    \item [$(ii)$] $I$ is $F$-injective if and only if $F(M, I)=0$, for all modules $M$ in $\CM.$
\end{itemize}
\end{corollary}

\begin{proof}
We only prove $(i)$. Proof of $(ii)$ is similar. If $P$ is in $\CP(F)$, then obviously for every $M$ in $\CM$, any $F$-exact sequence in $F(P, M)$ splits. So $F(P, M)=0$. The converse follows from the previous lemma.
\end{proof}

When $F$ has enough projective or enough injective objects, we have the following characterization of $F$-exact sequences in terms of $F$-projective and $F$-injective modules.

\begin{proposition}\label{F-n-proj-proper}
Let $\xi \colon 0 \lrt M' \lrt M^1 \lrt \cdots \lrt M^n\st{f} \lrt M \lrt 0$ be an $n$-exact sequence in $\CM$.
\begin{itemize}
\item [$(1)$] If $F$ has enough projectives, then $\xi$ is $F$-exact if and only if
\begin{align}
0 & \lrt \CM(P, M') \lrt \CM(P,M^1) \lrt \cdots \lrt \CM(P,M^n) \lrt \CM(P, M) \lrt 0\notag
\end{align}
is exact for every $F$-projective module $P$.
\item [$(2)$] If $F$ has enough injectives, then $\xi$ is $F$-exact if and only if
\begin{align}
0 & \lrt \CM(M, I) \lrt \CM(M^n, I) \lrt \cdots \lrt \CM(M^1, I) \lrt \CM(M', I) \lrt 0 \notag
\end{align}
is exact for every $F$-injective module $I$.
\end{itemize}
\end{proposition}

\begin{proof}
We only prove $(1)$. The proof of the second statement follows similarly. The `only if' part follows by definition. For the proof of the `if' part assume that for every $F$-projective module $P$, the sequence $\CM(P, \xi)$ is exact. Since $F$ has enough projectives, there exists an $F$-exact sequence
\begin{equation}
\eta \colon 0 \lrt L^1 \lrt L^2 \lrt \cdots \lrt  L^n \lrt Q \st{g} \lrt M \lrt 0 \notag
\end{equation}
such that $Q$ is $F$-projective. Since the sequence $\CM(Q, \xi)$ is exact, there is a morphism $s: Q \lrt M^n$ such that $fs=g$. Using the property of weak kernels we can construct the commutative diagram
\begin{equation}
\begin{tikzcd}
\eta \colon & 0 \rar & L^1 \rar \dar{d}& L^2 \rar \dar & \cdots\rar & L^n \rar \dar & Q \rar{g} \dar{s} & M \dar[equals] \rar & 0 \\
\xi \colon & 0 \rar & M' \rar & M^1 \rar & \cdots \rar & M^{n-1} \rar & M^n \rar{f} & M \rar & 0. \end{tikzcd} \notag
\end{equation}
By \cite[Proposition 4.8]{Ja}, this diagram is an $n$-pushout diagram of the $F$-exact sequence $\eta$ along the morphism $d$. Hence, $\xi$ is $F$-exact, as it is claimed.
\end{proof}

\begin{remark} \label{n-exactness} Let $\xi \colon 0 \lrt M' \lrt M^1\lrt \cdots \lrt M^n\st{f}\lrt M \lrt 0$ be an $n$-exact sequence in $\CM$. By \cite[Proposition 3.2.1]{I1}  for every $X \in \CM$, the sequence $\CM(X, \xi)$ is exact if and only if so is $\CM(\xi, \tau_n X)$. 
\end{remark}

\begin{corollary}
Let $F$ be a subfunctor of $\Ext^n_{\CM}( - , - )$ with enough injectives and enough projectives. The following equalities hold. 
\begin{itemize}
   \item [$(1)$] $\CI(F)=\tau_n(\CP(F)) \cup \CI(\La)$.
   \item [$(2)$] $\CP(F)=\tau^{-1}_n(\CI(F)) \cup \CP(\La)$.
\end{itemize}
\end{corollary}

\begin{proof}
We will provide proof for Statement $(1)$, and omit the proof of Statement $(2)$ as it is similar. We first show that $\CI(F) \subseteq \tau_n(\CP(F)) \cup \CI(\La)$. Let $X \in \CI(F)$. If $X \in \CI(\La)$, there is nothing to prove. So assume that $X$ is not injective. Since $X$ is $F$-injective, $\CM(\xi, X)$ is exact, for every $F$-exact sequence $\xi$. This, in view of Remark \ref{n-exactness}, implies that $\CM(\tau_n^{-1}X, \xi)$ is exact, for every $F$-exact sequence $\xi$. So $\tau_n^{-1}X \in \CP(F)$, or equally, $X \in \tau_n(\CP(F))$. So we have the desired inclusion. The reverse inclusion follows by using a similar argument. So the proof is complete.
\end{proof}

\begin{remark}
We say that a subcategory $\SX$ of an abelian category $\SA$ is of finite type if the number of iso-classes of indecomposable objects in $\SX$ is finite. This is equivalent to saying that it has an additive generator, i.e., there exists $X \in \SX$ such that $\SX=\add X$. 
The above corollary, in particular, implies that $\CP(F)$ is of finite type if and only if $\CI(F)$ is of finite type.
\end{remark}

\begin{remark}
It is known that the number of non-isomorphic indecomposable projective modules in $\mmod\La$ is equivalent to the number of non-isomorphic indecomposable injective modules. Suppose that $\CP(F)$ and $\CI(F)$ are of finite type. In that case, the above corollary extends this fact to the higher version by establishing equality between the number of non-isomorphic indecomposable modules in $\CP(F)$ and the number of non-isomorphic indecomposable modules in $\CI(F)$.
\end{remark}

\begin{notation}\label{notation2.10}
Let $\CN$ be a subcategory of $\CM.$ For every $M$ and $M'$ in $\CM,$ let $F^n_{\CN}(M, M')$ and $F^{n\CN}(M, M')$ denote the subsets
\[ \{ \xi \in \Ext^n_{\CM}(M, M') \mid \CM(X, \xi) \ \text{is exact for all } X \in \CN \}\]
and
\[\{ \xi \in \Ext^n_{\CM}(M, M') \mid \CM(\xi, X) \ \text{is exact for all } X \in \CN \}\]
of $\Ext^n_{\CM}(M, M')$, respectively.
\end{notation}

It is obvious that an $n$-exact sequence 
$$\xi \colon 0 \lrt M' \st{f^0}{\lrt} M^1 \lrt \cdots \lrt M^n \st{f^n}\lrt M \lrt 0$$ 
is in $F^n_{\CN}(M, M')$ if and only if for every $X \in \CN$, the induced morphism $\CM(X, f^n)$ is an epimorphism. Similarly, an $n$-exact sequence $\xi$ is in $F^{n\CN}(M, M')$ if and only if for every $X \in \CN$, the induced morphism  $\CM(f^0, X)$ is an epimorphism.

As in the Notation \ref{notation}, we omit $n$ in the above subfunctors and just write $F_{\CN}$ and $F^{\CN}$ to denote them.

\begin{proposition}\label{additivesubfunctor}
Let $\CN$ be a subcategory of $\CM$. Then $F_{\CN}( - , - )$ and $F^{\CN}( - , - )$ are additive subfunctors of $\Ext^n_{\CM}( - , - )$.
\end{proposition}

\begin{proof}
We prove that $F_{\CN}( - , - )$ is an additive subfunctor of $\Ext^n_{\CM}( - , - )$. The proof of this fact for $F^{\CN}( - , - )$ is similar and thus omitted. To this end, in view of \ref{ASadeghi}, we prove that $F_{\CN}( - , - )$ is closed under $n$-pullbacks and $n$-pushouts along any other morphism in $\CM$ and is furthermore closed under direct sums of $F$-exact sequences. Let $\xi: 0 \lrt M' \lrt M^1\lrt \cdots \lrt M^n\st{f}\lrt M \lrt 0$ be an $n$-exact sequence in $F_{\CN}(M, M')$ and let $h: X \lrt M$ be a morphism in $\CM$. Consider the $n$-pullback diagram
\begin{equation*}
\begin{tikzcd}
\xi h \colon  &0 \rar & M' \rar \dar[equal] & Y^1  \rar\dar & Y^2 \rar \dar & \cdots \rar &Y^n\rar{t}\dar{s} & X  \dar{h}\rar & 0   \\
\xi \colon &0 \rar & M' \rar & M^1\rar & M^2 \rar & \cdots\rar & M^n\rar{f}& M \rar  & 0
\end{tikzcd}
\end{equation*}
of $\xi$ along $h$. According to the construction of $n$-pullbacks \cite{Ja}, the rightmost square of the above diagram is obtained as depicted in the following diagram
\begin{equation*}
\footnotesize
\begin{tikzcd}
Y^{n} \ar{dd}{s} \ar{rr}{t} \drar{\gamma} && X  \ar{dd}{h} \\
& Z \urar{\alpha} \dlar{\beta} \\
M^n \ar{rr}{f} && M
\end{tikzcd}
\end{equation*}
where $(Z, \alpha, \beta)$ is the pullback of $f$ along $h$ in $\mmod \La$ and $\gamma: Y^n \lrt Z$ is a right $\CM$-approximation of $Z$. Let $g: N \lrt X$ be an arbitrary morphism in $\CM$, where $N \in \CN$. Since $\xi$ is in $F_{\CN}(M, M')$, there is a morphism $k:N \lrt M^n$ such that $fk=hg$. Hence, by the properties of pullback diagrams, there exists a map $l:  N \lrt Z$ such that $\alpha l=g$. Since $\gamma$ is a right $\CM$-approximation of $Z$, there exists a map $v: N \lrt Y^n$ such that $\gamma v=l$. This implies that $g$ factors through $t$ via $v$. So the induced morphism $\CM(N, t):\CM(N, Y^n)\lrt \CM(N, X)$ is surjective. Hence $\xi h$ is in $F_{\CN}(X, M')$.

To prove that $F_{\CN}$ is closed under $n$-pushouts, let $d: M' \lrt L$ be a morphism in $\CM$. Consider the $n$-pushout diagram
\begin{equation*}
\begin{tikzcd}
\xi  \colon  &0 \rar & M' \rar \dar{d} & M^1  \rar\dar & M^2 \rar \dar & \cdots \rar &M^n\rar{f}\dar{u} & M \dar[equal]\rar & 0   \\
\xi d \colon &0 \rar & L \rar & L^1\rar & L^2 \rar & \cdots\rar & L^n\rar{f'}& M \rar  & 0
\end{tikzcd}
\end{equation*}
of $\xi$ along $d$, where $\xi$ is in $F_{\CN}$. Let $g: N \lrt M$ be an arbitrary morphism in $\CM$ with $N \in \CN$. Since $\xi$ is in $F_{\CN}(M, M')$, there is a morphism $h: N \lrt M^n$ such that $fh=g$. So $g$ factors through $f'$ via $uh$, that is the morphism $\CM(N, f'):\CM(N, L^n)\lrt \CM(N, M)$ is surjective. Consequently, $\xi d$ is in $F_{\CN}(M, L)$.

Finally, since $\CM(X, - )$, for $X \in \CM$, commutes with finite direct sums, we deduce that $F_{\CN}( - , - )$ is closed under finite direct sums of $F$-exact sequences.
\end{proof}

\begin{remark}\label{Lemma1-partone}
Let $F$ be an additive subfunctor of $\Ext^n_{\CM}(-, -)$. It follows from Proposition \ref{F-n-proj-proper} that $F=F_{\CP(F)}$, provided $F$ has enough projectives and $F=F^{\CI(F)}$, provided $F$ has enough injectives.
\end{remark}

We have the following relationship between these two additive subfunctors.

\begin{proposition}\label{n-almost-proj}
Let $\CN$ be a subcategory of $\CM$. Then $F_{\CN}=F^{\tau_n\CN},$ where for a subcategory $\CX$ of $\CM$, $\tau_n \CX:=\{\tau_nX\mid X\in \CX\}$.
\end{proposition}

\begin{proof}
The proof is an immediate consequence of Remark \ref{n-exactness}.
\end{proof}

An $n$-almost split sequence 
\[\xi \colon 0 \lrt X^0 \lrt X^1 \lrt X^2 \lrt \cdots \lrt X^n \lrt X^{n+1} \lrt 0 \]
is called $n$-$F$-almost split if it is $F$-exact. For simplicity, we drop $n$ and refer to it as an $F$-almost split sequence.

\begin{proposition}\label{ind-non-proj-non-inj}
Let $X$ and $Y$ be, respectively, indecomposable non-projective and indecomposable non-injective modules in $\CM$. The following assertions hold.
\begin{itemize}
\item [$(1)$] If $F$ has enough projectives, then $X \in \CP(F)$ if and only if the $n$-almost split sequence
\[\xi \colon \ 0 \lrt \tau_nX \lrt M^1 \lrt \cdots \lrt M^n \lrt X \lrt 0 \]
is not $F$-exact.
\item [$(2)$] If $F$ has enough injectives, then $Y \in \CI(F)$ if and only if the $n$-almost split sequence
\[\eta \colon 0 \lrt Y \lrt M^1 \lrt \cdots \lrt M^n \lrt \tau^{-1}_nY \lrt 0\]
is not $F$-exact.
\end{itemize}
\end{proposition}

\begin{proof}
We will only provide proof for Statement (1) since Statement (2) has a dual proof. Assume that $X \in \CP(F)$. The $n$-almost split sequence $\xi$ in $\CM$ ending at $X$ can not be $F$-exact, as otherwise, it splits, a contradiction. Conversely, assume that the $n$-almost split sequence
\[\xi \colon 0 \lrt \tau_nX \lrt M^1 \lrt \cdots \lrt M^{n-1} \lrt M^n\st{f} \lrt X \lrt 0 \]
is not $F$-exact. If $X$ is not in $\CP(F),$ there is an $F$-exact sequence
\[\eta \colon  0 \lrt L^0 \lrt L^1 \lrt \cdots \lrt L^{n-1} \lrt P \st{g}{\lrt} X \lrt 0\]
with a non-split epimorphism $P\st{g}\lrt X \lrt 0$.  Since $f$ is a right almost split morphism, there is a morphism $h: P \lrt M^n$ that makes the diagram
\begin{equation*}
\begin{tikzcd}
P \rar{g}\dar{h} & X \dar[equal]\rar & 0   \\
 M^n\rar{f}& X \rar  & 0
\end{tikzcd}
\end{equation*}
 commutative. Using the property of weak kernels, we can complete the above square to the commutative diagram
\begin{equation*}
\begin{tikzcd}
 \eta \colon  &0 \rar & L^0 \rar \dar{d} & L^1  \rar\dar & L^2 \rar \dar & \cdots \rar & P \rar{g}\dar{h} & X  \dar[equal]\rar & 0   \\
 \xi \colon &0 \rar & \tau_nX \rar & M^1\rar & M^2 \rar & \cdots\rar & M^n\rar{f}& X \rar  & 0.
\end{tikzcd}
\end{equation*}
Hence the $n$-exact sequence $\xi$ is an $n$-pushout of the $F$-exact sequence $\eta$ along $d$. This implies that $\xi$ is an $F$-exact sequence, a contradiction.
\end{proof}

\begin{proposition}\label{proj-injectiv-objec}
Let $\CN$ be an additive subcategory of $\CM$. Then
\begin{itemize}
    \item [$(i)$] $\CP(F_{\CN})=\CN\cup \CP(\La)$.
    \item [$(ii)$] $\CI(F^{\CN})=\CN\cup \CI(\La)$.
\end{itemize}
\end{proposition}

\begin{proof}
It is clear that $\CN\cup \CP(\La)\subseteq \CP(F_{\CN})$. Let $X$ be an indecomposable module that does not belong to $\CN \cup \CP(\La)$. We show that $X$ is not included in $\CP(F_{\CN})$. This completes the proof. Let
\[\xi \colon 0 \lrt \tau_nX \lrt M^1 \lrt \cdots \lrt M^{n-1} \lrt M^n\st{f} \lrt X \lrt 0 \]
be the $n$-almost split sequence in $\CM$ ending at $X$. Since $X$ is not in $\CN$, every morphism $g: N \lrt X$ with $N \in \CN $ is not a spilt epimorphism. The property of $f$ being right almost split implies that $g$ factors through $f$. Hence $\xi$ is an $F_{\CN}$-exact sequence. Therefore Proposition \ref{ind-non-proj-non-inj} yields that $X$ is not in $\CP(F_{\CN})$, as desired.
\end{proof}

We also need the following lemma that is of independent interest.

\begin{lemma}\label{cova-contra-tau}
Let $\CN$ be an additive subcategory of $\CM$. Then
\begin{itemize}
    \item [$(1)$] $\CN \cup \CP(\La)$ is a covariantly finite subcategory of $\CM$ if and only if so is $\tau_n \CN \cup \CI(\La)$.
    \item [$(2)$] $\CN \cup \CP(\La)$ is a contravariantly finite subcategory of $\CM$ if and only if so is $\tau_n \CN \cup \CI(\La)$.
\end{itemize}
\end{lemma}

 \begin{proof}
We will only consider the `only if' part of the statement (1), as the proofs for the other implication and for statement (2) are similar. Assume that $\tau_n \CN \cup \CI(\La)$ is a covariantly finite subcategory of $\CM$. Set $\CU:= \tau_n \CN$. It is obvious that $\overline{\CU}$ is a covariantly finite subcategory of $\overline{\CM}$. The equivalence $\tau_n:\underline{\CM}\lrt \overline{\CM}$ will be restricted to the equivalence $\tau_n| : \underline{\CN} \lrt \overline{\CU}$, hence the diagram
\[ \xymatrix{ \underline{\CM}\ar[r]^{\tau_n} &  \overline{\CM} \\
	\underline{\CN}	\ar@{^(->}[u] \ar[r]^{\tau_n|} & \overline{\CU} \ar@{^(->}[u]}\]
is commutative. The equivalence $\tau_n|$, in particular, implies that $\underline{\CN}$ is a covariantly finite subcategory of $\underline{\CM}$. We use this fact to show that $\CN \cup \CP(\La)$ is a covariantly finite subcategory of $\CM$. To do this, let $M \in \CM$ be an arbitrary element. By considering $\underline{M}$ as an object in $\underline{\CM}$, we deduce that there exists a left $\underline{\CN}$-approximation $\underline{f}:  \underline{M} \lrt \underline{X}$ in $\underline{\CM}$, where $\underline{f}$ is the residue class of a morphism $f: M \lrt X$ in $\CM$.
Let $e: M \lrt P$ be a left $\CP(\La)$-approximation of $M$. We claim that $\left[\begin{smallmatrix}f\\ e \end{smallmatrix}\right]:M \lrt X\oplus P$ is a left $\CX \cup \CP(\La)$-approximation of $M$. To prove the claim, let $g: M \lrt M'$ be a morphism with $M' \in \CN \cup \CP(\La)$. Passing to the stable category, we deduce that there exists a morphism $\underline{h}: \underline{M} \lrt \underline{M'}$ such that $\underline{h}\underline{f}=\underline{g}$. This, in particular, implies that $g - hf$ factors through a projective module $Q$, i.e., there exists morphisms $u: M \lrt Q$ and $\ell: Q \lrt M'$ making the diagram
\[\begin{tikzcd}
M \ar{rr}{g-hf} \drar{u} && M'   \\
& Q \urar{\ell}
\end{tikzcd}\]
commutative. On the other hand, since $e: M \lrt P$ is a left $\CP(\La)$-approximation of $M$, there exists $w: P \lrt Q$ such that $w e = u$. Now it is easy to see that the diagram
\begin{equation*}
\begin{tikzcd}
M\rar{\left[\begin{smallmatrix}f\\ e \end{smallmatrix}\right]}\dar{g} & X\oplus P\dlar{[h~~\ell w]}   \\ M'&
\end{tikzcd}
\end{equation*}
is commutative. This proves the claim and hence completes the proof.
\end{proof}

Now we have the necessary background to prove the main theorem of this section. 

\begin{theorem}\label{main-Theorem}
The maps $F\mapsto \CP(F)$ and $\CN \mapsto F_{\CN}$ induce mutually inverse bijections between the following two classes.
    \begin{itemize}
        \item [$(A)$] Additive subfunctors $F$ of $\Ext^n_{\CM}( - , - )$ with enough projectives.
        \item [$(B)$] Contravariantly finite subcategories $\CN$ of $\CM$ containing $\CP(\La)$.
    \end{itemize}
Moreover, these maps are restricted to mutually inverse bijections between the following two classes.
 \begin{itemize}
        \item [$(a)$] Additive subfunctors $F$ of $\Ext^n_{\CM}(-, -)$ with enough projectives and enough injectives.
        \item [$(b)$] Functorially finite subcategories $\CN$ of $\CM$  containing $\CP(\La)$.
    \end{itemize}
\end{theorem}

\begin{proof}
Let $F$ be an additive subfunctor of $\Ext^n_{\CM}( - , - )$ with enough projectives. Let $M \in \CM$. Since $F$ has enough projections, there exists an $F$-exact sequence 
\[\xi \colon \ 0 \lrt M^1 \lrt M^2 \lrt \cdots \lrt M^n \lrt P \lrt M \lrt 0,\]
with $P \in \CP(F)$. Now Proposition \ref{F-n-proj-proper} implies that $P \lrt M$ is a right $\CP(F)$-approximation of $M$. Hence $\CP(F)$ is a contravariantly finite subcategory of $\CM$ that clearly contains $\CP(\La)$.

On the other hand, let $\CN \subseteq \CM$ be a contravariantly finite subcategory that contains $\CP(\La)$. In view of Proposition \ref{additivesubfunctor} we only show that $F_{\CN}$ has enough projections. 

Let $M \in \CM$. There exists a left $\CN$-approximation $f: N \lrt M$ which is an epimorphism, because $\CN$ contains $\CP(\La)$. Consider an $n$-kernel of $f$ in $\CM$ to get the $n$-exact sequence
\[\xi \colon 0 \lrt M^1 \lrt M^2 \lrt \ldots \lrt M^n \lrt N \st{f}{\lrt} M \lrt 0.\]
Since $f$ is a left $\CN$-approximation, $\xi$ is obviously $F_{\CN}$-exact and $N$ is $F_{\CN}$-projective. Hence $F_{\CN}$ has enough projective objects. 
Now Remark \ref{Lemma1-partone} and Proposition \ref{proj-injectiv-objec} imply that these operations are mutually inverse.

For the `Moreover' part, assume that the subfunctor $F$ has enough projectives and enough injectives. Since $F$ has enough projectives, by Remark \ref{Lemma1-partone} and Proposition \ref{n-almost-proj}, $F=F_{\CP(F)}=F^{\tau_n\CP(F)}$. Since $F$ has enough injectives, in view of Propositions \ref{F-n-proj-proper} and \ref{proj-injectiv-objec}, we deduce that $\tau_n\CP(F)\cup\CI(\La)$ is a covariantly finite subcategory of $\CM$. Lemma \ref{cova-contra-tau} implies that $\CP(F)=\CP(F)\cup\CP(\La)$ is also covariantly finite in $\CM$. Hence  $\CP(F)$ is functorially finite. Therefore, the map $F \mapsto\CP(F)$ restricts to a map from $(a)$  to $(b)$. Conversely, let $\CN$ be a functorially finite subcategory of $\CM$. So, by the first part, $F:=F_{\CN}$ has enough projectives. It remains to show that $F$ has enough injectives as well. By Proposition \ref{n-almost-proj}, $F=F_{\CN}=F^{\tau_n\CN}=F^{\tau_n\CN \cup \CI(\La)}$. By Lemma \ref{cova-contra-tau}, $\tau_n\CN \cup \CI(\La)$ is a covariantly finite subcategory of $\CM$. Using a dual version of the first part implies that $F=F^{\tau_n\CN \cup \CI(\La)}$ has enough injectives. Hence the map $\CN \mapsto F_{\CN}$ restricts from $(b)$ to $(a)$. This completes the proof of the second part.
\end{proof}

As a corollary, we can record the following.

\begin{corollary}\label{finittype-enoughproje-enoughinject}
Let $\CM$ be an $n$-cluster tilting subcategory of $\mmod\La$ of finite type. Let $F$ be a subfunctor of $\Ext^n_{\CM}(-, -)$.  The following conditions are equivalent.
\begin{itemize}
\item[$(1)$] $F$ has enough projectives.
\item [$(2)$] $F$ has enough injectives.
\item [$(3)$] $F$ has enough projectives and enough injectives.
\end{itemize}
\end{corollary}

\begin{proof}
As $\CM$ is of finite type, every subcategory is functorially finite. In particular, $\CP(F)$ and $\CI(F)$ are functorially finite in $\CM$. The result then follows from Theorem \ref{main-Theorem} and its dual.
\end{proof}

\section{Relative monomorphism categories}\label{Section 3}
The morphism category of $\CM$, denoted by $\CH(\CM)$, is a category whose objects are morphisms in $\CM$ and whose morphisms are given by commutative diagrams. We let $\CS(\CM)$, resp. $\CF(\CM)$, denote the subcategory of $\CH(\La)$ consisting of all monomorphisms, resp. epimorphisms, in $\CM$.

In this section, we first recall the functors $\Psi: \CS(\CM) \lrt \mmod\underline{\CM}$ and $\Psi':  \CF(\CM) \lrt \overline{\CM}\mbox{-}{\rm{mod}}$ from \cite{AHS} and then study their relative versions. Let us begin by recalling the notion of objective functors. 

\s \label{objective} An additive functor $F : \SX \lrt \SY$ between additive categories $\SX$ and $\SY$ is called objective if any morphism $f$ in $\SX$ with $F(f) = 0$ factors through an object $K$ with $F(K) = 0$. It is known \cite[Appendix]{RZ} that if $F: \SX \lrt \SY$ is full, dense and objective, it induces an equivalence $\overline{F}: \SX/ \SK  \lrt \SY$, where $\add\SK$ is the class of all kernel objects of $F$. The composition of objective functors is not necessarily objective but if, in addition, we know that they are both full and dense, then their composition is objective, full and dense.

\s {\sc The functor $\Psi: \CS(\CM) \lrt \mmod\underline{\CM}$. } \label{Psi}
Let $(M_1 \stackrel{f}\lrt M_2)$ be an object of $\CS(\CM)$. By taking an $n$-cokernel of $f$, we obtain the $n$-exact sequence
\[0 \lrt M_1 \st{f}{\lrt} M_2 {\lrt} M^1 {\lrt} M^2  \lrt \cdots \lrt M^{n-1} {\lrt} M^n \lrt 0 \]
in $\CM$. Hence the following induced sequence
\[0 \lrt \CM( - , M_1) \lrt \CM( - , M_2) \lrt  \cdots \lrt \CM( - , M^{n-1}) \lrt \CM( - , M^n) \lrt F \lrt 0 \]
is exact on $\CM$, where $F$ is the cokernel of the morphism $\CM( - , M^{n-1}) \lrt \CM( - , M^n)$. Clearly, $F$ vanishes on projective modules and so $F \in \mmod\underline{\CM}$.
We define a functor \[\Psi: \CS(\CM) \lrt \mmod\underline{\CM} \] by setting $\Psi(M_1 \stackrel{f}\lrt M_2)=F.$ To see the act of $\Psi$ on morphisms see \cite{AHS}, where it is also shown that $\Psi$ is a well-defined functor, independent of the choice of $n$-cokernels.

\begin{theorem}(\cite[Theorem 4.1 and Corollary 4.2]{AHS})\label{Th-Psi}
The functor $\Psi:\CS(\CM) \longrightarrow \mmod\underline{\CM}$ is full, dense, and objective. Furthermore, there exists an equivalence $\overline{\Psi}: {\CS(\CM)}/{\CV}\longrightarrow \mmod\underline{\CM}$ of categories that makes the diagram
\begin{equation*}
\begin{tikzcd}
\CS(\CM)\rar{\pi}\dar{\Psi} & {\CS(\CM)}/{\CV}\dlar{\overline{\Psi}}   \\
 \mmod\underline{\CM}&
\end{tikzcd}
\end{equation*}
commutative, where $\CV$ is the full subcategory of $\CS(\CM)$ generated by all finite direct sums of objects of the form $(M\st{1}\lrt M)$ and $(0\lrt M)$, where  $M$ runs over objects of $\CM$.
\end{theorem}

\s {\sc The functor $\Psi':  \CF(\CM) \lrt \overline{\CM}\mbox{-}{\rm{mod}}$.} \label{Psi'}
Pick an epimorphism $(M_1 \st{f}\lrt M_2)$ in $\CF(\CM)$. By taking $n$-kernel in $\CM$ we get an $n$-exact sequence
\[0 \lrt M^1\st{d^1}\lrt M^2\st{d^2}\lrt  \cdots \lrt M^n \st{d^n}\lrt M_1 \st{f}\lrt M_2 \lrt 0.\]
This, in turn, induces the exact sequence
\[0 \lrt \CM(M_2, - ) \lrt \CM(M_1, - ) \lrt \cdots \lrt \CM(M^2, - ) \lrt \CM(M^1, - ) \lrt F \lrt 0.\]
Define $\Psi'(M_1 \st{f}\lrt M_2):=F.$ We refer the reader to \cite{AHS} for the details of the functor $\Psi'$.

\begin{theorem}(\cite[Theorem 6.3]{AHS})\label{duality-epi}
The functor $\Psi':\CF(\CM) \longrightarrow \overline{\CM}\mbox{-}\rm{mod}$ is full, dense, and objective. In particular, there exists a duality $\overline{\Psi'}: {\CF(\CM)}/{\CV'} \longrightarrow \overline{\CM}\mbox{-}{\rm mod}$ of categories that makes the diagram
\begin{equation*}
\begin{tikzcd}
\CF(\CM)\rar{\pi}\dar{\Psi'} & {\CF(\CM)}/{\CV'}\dlar{\overline{\Psi'}}   \\
\overline{\CM}\mbox{-}{\rm mod}&
\end{tikzcd}
\end{equation*}
commutative, where $\CV'$ is the full subcategory of $\CF(\CM)$ generated by all finite direct sums of objects of the form $(M\st{1}\lrt M)$ and $(M\lrt 0)$, where  $M$ runs over objects of $\CM$.
\end{theorem}

\s {\sc The relative versions of the functors $\Psi$ and $\Psi'$.} \label{PsiandPsi'}
To study the relative versions of the functors $\Psi$ and $\Psi'$ defined above, we need to introduce some definitions and notations.

\s {\bf Definitions and Notations.} Let $\CM$ be an $n$-cluster tilting subcategory of $\mmod\La$ and $F$ be a subfunctor of $\Ext^n_{\CM}( - , - )$.

$(i)$ We say that a functor $F \in \mmod\CM$ is presented by an $F$-epimorphism if there exists an $F$-exact sequence
\[0 \lrt M^0 \lrt M^1 \lrt \ldots \lrt M^n \lrt M^{n+1} \lrt 0,\]
such that $F=\Coker(\CM( - ,M^n) \lrt \CM( - ,M^{n+1})).$ Dually, we define functors that are co-presented by an $F$-monomorphism. Denote by $\mmodd\CM $, resp. $\CM\mbox{-}{\rm mod}_F$, the subcategory of $\mmod\CM$ consisting of all functors which are presented, resp. co-presented, by an $F$-epimorphism, resp. an $F$-monomorphism.

$(ii)$ Let $\underline{\CM}_F$  denote the additive quotient category $\CM/\CP(F)$ of $\CM$ with respect to $\CP(F)$.  Dually, $\overline{\CM}_F$  denotes the additive quotient category $\CM/\CI(F)$ of $\CM$ with respect to $\CI(F)$. Denote by $\mmod \underline{\CM}_F$, resp.
$\overline{\CM} _F\mbox{-}{\rm mod}$, the category of finitely presented right $\underline{\CM}_F$-modules, resp. left $\overline{\CM}_F$-modules.

\begin{remark}
The canonical functor $\pi:\CM \lrt \underline{\CM}_F$ induces a fully faithful functor
\[\mmod\underline{\CM}_F \lrt \mmod\CM.\]
Its essential image is the collection of all functors of $\mmod\CM$ that vanish on $\CP(F)$. Hence we identify $\mmod\underline{\CM}_F$ with this subcategory of $\mmod\CM$. Note that under this identification $\mmodd\CM \subseteq \mmod\underline{\CM}_F$.
\end{remark}

\begin{proposition}\label{Prop4.5-enoughproj}
Let $F$ be a subfunctor of $\Ext^n_{\CM}(-, -)$. If $F$ has enough projectives, then $\mmodd\CM =\mmod \underline{\CM}_F$.
\end{proposition}

\begin{proof}
In view of the above remark, for the proof we only need to show that $$\mmod\underline{\CM}_F \subseteq \mmodd\CM .$$ Let $G \in \mmod\underline{\CM}_F$ be a functor and $\underline{\CM}_F(-, X)\st{f_*} \lrt \underline{\CM}_F(-, Y)\lrt G \lrt 0$ be a projective presentation of it. The natural transformation $f_*$ is induced by the residue class of a morphism $f:X\lrt Y$ in $\CM$. By assumption, there is an $F$-epimorphism $h:P\lrt Y$, with $P \in \CP(F)$, that sits in an $F$-exact sequence $\xi$. Taking $n$-pull back of $\xi$ along morphism $f$ gives rise to the $F$-epimorphism $[f~~h]: X\oplus P\lrt Y$. This morphism induces a projective presentation of $G$ in $\mmodd\CM $, which means $G \in \mmodd\CM .$
\end{proof}

Dually we have the following statement. Its proof is similar to the proof of the above proposition and so we state it without proof.

\begin{proposition}\label{Prop4.6-enoughinjective}
Let $F$ be a subfunctor of $\Ext^n_{\CM}(-, -)$. If $F$ has enough injectives, then $\CM\mbox{-}{\rm mod}_F=\overline{\CM}_F\mbox{-}{\rm mod}$.
\end{proposition}

Let $F$ be an additive subfunctor of $\Ext^n_{\CM}( - , - )$. 
Let $\CS_F(\CM)$, resp. $\CF_F(\CM)$, denote the subcategory of $\CS(\La)$, resp. $\CF(\La)$, consisting of all $F$-monomorphisms, resp. all $F$-epimorphisms. Note that $\CS_F(\CM) =\CS(\CM)$, resp. $\CF_F(\CM) = \CF(\CM)$, when $F=\Ext^n(-, -)$.\\

It follows from definition that the restrictions of the functors $\Psi$ and $\Psi'$ to $\CS_F(\CM)$ and $\CF_F(\CM)$, respectively, give rise to the functors
\begin{itemize}
\item [-] $\Psi_F:\CS_F(\CM)\lrt \mmodd\CM$; and
\item [-] $\Psi'_F:\CF_F(\CM)\lrt \CM\mbox{-} {\rm mod}_F $
\end{itemize}

\begin{theorem}\label{equivalece-three}
The functors $\Psi_F$ and $\Psi'_F$  are full, dense and objective. In particular, they induce the equivalences
\begin{itemize}
    \item [-] $\overline{\Psi}_F:\CS_F(\CM)/\CV\lrt \mmodd\CM$; and
    \item [-] $\overline{\Psi'}_F:\CF_F(\CM)/\CV'\lrt \CM\mbox{-} {\rm mod} _F $,
\end{itemize}
where $\CV$, resp. $\CV'$, is the full subcategory of $\CS(\CM)$, resp. $\CF(\CM)$, defined in Theorem \ref{Th-Psi}, resp. Theorem \ref{duality-epi}.
\end{theorem}

\begin{proof}
We will only prove the theorem for $\Psi_F$, as the argument for the proof for $\Psi'_F$ is similar. The connection between the functors $\Psi_F$  and $\Psi$ is depicted in the commutative diagram
\[\xymatrix{\CS_F(\CM) \ar@{=}[d]\ar@{^(->}[r]&\CS(\CM) \ar[r]^{\Psi}& \mmod \underline{\CM} \\\CS_F(\CM)\ar[rr]^{\Psi_F} &  & \mmodd\CM .\ar@{^(->}[u] }\]
Since split (contractible) complexes of length $n+2$ are $F$-exact, $\CV$ is a subcategory of $\CS_F(\CM)$. By Theorem \ref{Th-Psi}, $\Psi$ is full and objective, hence the above diagram implies that $\Psi_F$ is also full and objective. The density of $\Psi_F$ follows from the definition of $\mmodd\CM$. The fact that $\overline{\Psi}_F$ is an equivalence follows from the properties of the objective functors, see \ref{objective}.
\end{proof}

\s {\sc The functors ${\Phi_F}: \CE_F(\CM)  \lrt \mmodd \CM$ and $\Phi'_F: \CE_F(\CM) \lrt (\CM_F \mbox{-}{\rm mod})^{\rm op} .$}\label{Two functors PHI}

Let $\CE_F(\CM)$ be the category whose objects are $F$-exact sequences
\[0 \lrt M^0\lrt M^1\lrt \cdots \lrt M^n\lrt M^{n+1}\lrt 0 \]
in $\CM$ and whose morphisms are given by morphisms of complexes. We will drop $F$ in the symbol $\CE_F(\CM) $ whenever $F=\Ext^n_{\CM}(-, -)$. Denote by $\widetilde{\CE_F}(\CM)$ the homotopy category of $\CE_F(\CM)$, that is, the quotient category of $\CE_F(\CM)$ with
respect to the ideal of  null-homotopic morphisms. Note that two morphisms $f, g:{\xi} \lrt {\xi'} $ in $\CE_F(\CM)$ are homotopy equivalent if and only if $f-g:{\xi}\lrt {\xi'} $ factors through a direct sum of the contractible $n$-exact sequences of the type
\[A_i(M) \colon 0 \lrt 0 \lrt \cdots \lrt M \st{1}\lrt M \lrt \cdots \lrt 0 \lrt 0, \]
where $M \in \CM$ and the left hand $M$ is in the $i$th position, where $0 \leq i \leq n.$ Consequently, the homotopy category $\widetilde{\CE_F}(\CM)$ is equivalent to the additive quotient category $\CE_F(\CM)/\CW$, where $\CW$ is the the full subcategory of $\CE_F(\CM)$ generated by all finite direct sums of contractible $n$-exact sequences of the form $A_i(M)$, where $M$ runs over objects of $\CM$.

\s Let  $$\sigma_F: \CE_F(\CM)  \lrt \CS_F(\CM)$$
be the (hard) truncation functor, which is defined by sending an $F$-exact sequence ${\xi} \in \CE_F(\CM)$ to its first two terms, that is,
\begin{equation*}
\begin{tikzcd}
{\xi}\dar[mapsto] & 0 \rar & M^0 \rar \dar[equal] & M^1  \rar\dar[equal] & M^2 \rar  & \cdots \rar &M^n\rar & M^{n+1} \rar & 0   \\
\sigma_F {\xi} &  & M^0 \rar & M^1 &  &  & &   &
\end{tikzcd}
\end{equation*}

We define the functor ${\Phi_F}: \CE_F(\CM)  \lrt \mmodd \CM$ to be the composition of the functors
\[\Phi_F \colon \CE_F(\CM)\st{\sigma_F}\lrt \CS_F(\CM) \st{\Psi_F}\lrt \mmodd\CM.\]

\begin{theorem}\label{equiv-n-first}
The functor $\Phi_F:\CE_F(\CM)\lrt \mmodd\CM$ is full, dense and objective. In particular, it induces an equivalence $\widetilde{\Phi}_F: \widetilde{\CE_F}(\CM) \lrt \mmodd\CM$ that makes the diagram
\begin{equation*}
\begin{tikzcd}
\CE_F(\CM)\rar{\pi}\dar{\Phi_F} & \widetilde{\CE_F}(\CM)\dlar{\widetilde{\Phi}_F}   \\
 \mmodd \CM&
\end{tikzcd}
\end{equation*}
commutative.
\end{theorem}

\begin{proof}
It follows from the definition that $\sigma_F$ is full, dense and objective. On the other hand, by Theorem \ref{equivalece-three}, $\Psi_F$ is also full, dense and objective. Hence so is their composition, i.e., $\Phi_F$ is full, dense and objective. Now it is easy to see that the kernel objects of $\Phi_F$ are additively generated by the $F$-exact sequences of the type $A_i(M)$, where $M \in \CM$. Hence $\widetilde{\Phi}_F$ is an equivalence.  
\end{proof}

\s By the dual construction, we define $\sigma'_F:\CE_F(\CM)\lrt \CF_F(\CM)$ as follows
\begin{equation*}
\begin{tikzcd}
{\xi} \dar[mapsto] & 0 \rar & M^0 \rar & M^1\rar &M^2\rar &\cdots \rar & M^n \rar \dar[equal] & M^{n+1}\rar \dar[equal] & 0\\
\sigma'_F {\xi}  &  &  &  &   &  &M^n\rar & M^{n+1}  &    
\end{tikzcd}
\end{equation*}

In view of this, we define $\Phi'_F:\CE_F(\CM) \lrt \CM\mbox{-}{\rm mod}_F$ to be the composition of the following functors
$$\Phi'_F:\CE_F(\CM)\st{\sigma'_F}\lrt \CF_F(\CM) \st{\Psi'_F}\lrt \CM\mbox{-}{\rm mod}_F$$

\begin{theorem}\label{duality-n-exat2}
The functor $\Phi'_F:\CE_F(\CM)\lrt \CM\mbox{-}{\rm mod}_F$ is full, dense and objective. In particular, it induces a duality $\widetilde{\Phi'}_F$ satisfying the following commutative diagram

\begin{equation*}
\begin{tikzcd}
\CE_F(\CM)\rar{\pi}\dar{\Phi'_F} & \widetilde{\CE_F}(\CM) \dlar{\widetilde{\Phi'}_F}  \\
\CM\mbox{-}{\rm mod}_F &
\end{tikzcd}
\end{equation*}
\end{theorem}

\begin{proof}
The proof is similar to the proof of Theorem \ref{equiv-n-first} so we omit it.
\end{proof}

The equivalences given in Theorems \ref{equivalece-three}, \ref{equiv-n-first} and \ref{duality-n-exat2} are related to each other as depicted in the following diagram
\begin{equation}\label{allfunctors}
\begin{tikzcd}
 &\CF_F(\CM)/\CV'\dlar{\overline{\Psi'}_F} &  &   \\
  (\CM\mbox{-}{\rm mod}_F)^{\rm op} &\widetilde{\CE_F}(\CM)\rar{\widetilde{\Phi}_F}\uar{\widetilde{\sigma'}_F}\dar{\widetilde{\sigma}_F} \lar{\widetilde{\Phi'}_F}& \mmodd\CM   &  \\
   & \CS_F(\CM)/\CV\urar{\overline{\Psi}_F} &  & \\
   & &  &
\end{tikzcd}
\end{equation}
where the equivalences $\widetilde{\sigma}_F$ and $\widetilde{\sigma'}_F$ are induced by the truncation functors $\sigma$ and $\sigma'$, respectively.

\begin{remark}
The equivalences presented in the middle column of the diagram \eqref{allfunctors} can be considered as higher analogs of the pair of inverse equivalences
\[{\rm Ker}:\CF(\La) \lrt \CS(\La) \ \ \text{and} \ \ {\rm Cok}:\CS(\La)\lrt \CF(\La)\]
which are the restrictions of the kernel and cokernel functors, see \cite[Lemma 1.2]{RS}.
\end{remark}

\section{Relative $n$-Auslander-Reiten duality}\label{Section 4}
In this section, we provide a relative Auslander-Reiten duality formula in the higher setting. 
To begin, in view of the Diagram \eqref{allfunctors}, we set $\Sigma_F$ to be the composition
\[\Sigma_F: \mmodd\CM \st{{\widetilde{\Phi}_F}^{-1}}\lrt  \widetilde{\CE_F}(\CM)\st{\widetilde{\Phi'}_F}\lrt \CM\mbox{-}{\rm mod}_F\]
of functors, which is a duality.

\begin{lemma}\label{Lemma-Sigma}
Let $M\in \CM$. The following statements hold.
\begin{itemize}
    \item [$(1)$] If $F$ has enough projectives, then  $\Sigma_F(\underline{\CM}_F( - , M))\simeq F(M, - ).$
    \item [$(2)$] If $F$ has enough injectives, then $\Sigma^{-1}_F(\overline{\CM}_F(M, - ))\simeq F( - , M).$
\end{itemize}
\end{lemma}

\begin{proof}
We only prove the validity of the Statement $(1)$ as the proof for the Statement $(2)$ is similar. Since $F$ has enough projectives, by Proposition \ref{Prop4.5-enoughproj}, $\mmodd\CM =\mmod \underline{\CM}_F$. Moreover, there exists an $F$-exact sequence
\[\xi \colon 0 \lrt M^0 \lrt M^1 \lrt \cdots \lrt M^n \lrt P \lrt M \lrt 0\]
with $P \in \CP(F)$. This, in turn, induces the exact sequence
\[0 \lrt \CM( - , M^0) \lrt \CM( - , M^1) \lrt \cdots \lrt \CM( - , P) \lrt \CM( - , M) \lrt \underline{\CM}_F( - , M) \lrt 0\]
in $\mmod \CM$. Hence $\Sigma_F(\underline{\CM}_F(-, M))=\widetilde{\Phi'}_F\circ \widetilde{\Phi}^{-1}_F(\underline{\CM}_F(-, M))=\widetilde{\Phi'}_F(\xi).$
On the other hand, by an argument similar to \cite[Proposition 2.2]{JK}, we have the following exact sequence
\[0 \lrt  \CM(M, - )\lrt \CM(P, - )\lrt \cdots \lrt \CM(M^0, - ) \lrt F(M, - ) \lrt F(P, - ).\]
But $F(P, - )=0$, because $P \in \CP(F)$. Thus $\widetilde{\Phi'}_F(\xi)=F(M, - )$. This completes the proof.
\end{proof}

\begin{remark}
An additive $R$-linear essentially small category $\CX$ is called a dualizing $R$-variety if the functor $\Mod\CX \lrt \Mod\CX^{\op}$ taking $F$ to $DF$, induces a duality $\mmod\CX \lrt \mmod\CX^{\op}$. By \cite[Theorem 2.4]{AR1}, if $\CX$ is a dualising $R$-variety, then $\mmod\CX$ and $\mmod\CX^{\op}$ are abelian categories with enough projectives and injectives.

Since $\CM$ is an $n$-cluster tilting subcategory of $\mmod\La$, it is functorially finite. So by \cite[Theorem 2.3]{AS2}, it is a dualizing variety. That is, there exists a duality $D: \mmod\CM \lrt \CM \mbox{-}{\rm mod}$, which is defined by using the standard duality $D:\mmod\La\lrt \mmod\La^{\rm op}.$ It easily follows from the definition that we have the following restrictions
\[ \xymatrix{ \mmod \CM\ar[r]^{D} &  \CM \mbox{-} {\rm mod} \\
	\mmod \overline{\CM}\ar@{^(->}[u] \ar[r]^{D} & \overline{\CM}\mbox{-}{\rm mod} \ar@{^(->}[u]\\
		\mmod \overline{\CM}_F\ar@{^(->}[u] \ar[r]^{D} & \overline{\CM}_F\mbox{-}{\rm mod} \ar@{^(->}[u]}\]
where, by abuse of notation, we use the same notation `$D$' for all induced restrictions. The above diagram, in particular, implies that $\overline{\CM}$ and $\overline{\CM}_F$ are dualizing varieties.
\end{remark}

\begin{theorem}\label{Theroem-relative-tau}
Let $F$ be a subfunctor of $\Ext^n_{\CM}(-, -)$ with enough projectives and enough injectives. There exists a, unique up to isomorphism, equivalence $\tau_{ F}:\underline{\CM}_F\lrt \overline{\CM}_F$ such that for every $X, Y \in \CM$ induces the natural isomorphisms
\[D\underline{\CM}_F(\tau^{-1}_{ F}Y, X)\simeq F(X, Y)\simeq D\overline{\CM}_F(Y, \tau_{ F}X).\]
\end{theorem}

\begin{proof}
By Propositions \ref{Prop4.5-enoughproj} and \ref{Prop4.6-enoughinjective}, we infer that
\[\mmodd\CM=\mmod \underline{\CM}_F  \ \ \ {\rm and} \ \ \ \CM\mbox{-}{\rm mod}_F=\overline{\CM}_F\mbox{-}{\rm mod}.\]

Consider the composition
\[D\circ \Sigma_F:\mmod \underline{\CM}_F\st{\Sigma_F}\lrt \overline{\CM}_F\mbox{-}{\rm mod}\st{D}\lrt \mmod \overline{\CM}_F \]
of functors. Since $\CM$ is closed under direct summands, $\underline{\CM}_F$ and $\overline{\CM}_F$ are idempotent complete. Hence Yoneda functor induces the equivalences
\[P_{(-)}: \underline{\CM}_F \simeq {\rm prj}\mbox{-}\mmod \underline{\CM}_F \ \ \ {\rm and} \ \ \ P^{(-)}: \overline{\CM}_F\simeq {\rm prj}\mbox{-}\mmod \overline{\CM}_F. \]
Since an equivalence between abelian categories can be restricted to their projective objects, we get the equivalence
\begin{equation}\label{eq1}
\tau_{F}:\underline{\CM}_F\st{P_{(-)}}\lrt {\rm prj}\mbox{-}\mmod \underline{\CM}_F\st{D\circ \Sigma_F|}\lrt {\rm prj}\mbox{-}\mmod \overline{\CM}_F \st{P^{{(-)}^{-1}}}\lrt \overline{\CM}_F 
\end{equation}
of functors. 

Since a duality sends a projective object to an injective one, the dualities $$D:\underline{\CM}_F\mbox{-}{\rm mod}\lrt \mmod \underline{\CM}_F  \ \ \ {\rm and} \ \ \ D:\overline{\CM}_F\mbox{-}{\rm mod}\lrt \mmod \overline{\CM}_F$$ 
give rise to the equivalences
\[\Theta:\underline{\CM}_F\st{P_{(-)}}\lrt {\rm prj}\mbox{-} \underline{\CM}_F\mbox{-}{\rm mod}\st{D}\lrt {\rm inj}\mbox{-}\mmod \underline{\CM}_F, \ X \mapsto D\underline{\CM}_F(X, -), \]
and
\[\Theta':\overline{\CM}_F\st{P^{(-)}}\lrt {\rm prj}\mbox{-} \overline{\CM}_F\mbox{-}{\rm mod}\st{D}\lrt {\rm inj}\mbox{-}\mmod \overline{\CM}_F, \ X \mapsto D\overline{\CM}_F(X, -). \]

The equivalence $\Sigma^{-1}_F\circ D$ is restricted to an equivalence between the subcategories of injective functors. This yields the equivalence 
\begin{equation}\label{eq2}
\tau'_{F}: \overline{\CM}_F\st{\Theta'}\lrt {\rm inj}\mbox{-}{\rm mod} \mbox{-}\overline{\CM}_F\st{\Sigma^{-1}_F\circ D}\lrt {\rm inj}\mbox{-}{\rm mod} \mbox{-}\underline{\CM}_F \st{\Theta^{-1}}\lrt  \underline{\CM}_F.
\end{equation}
We show that the functors $\tau_F$ and $\tau'_F$ are mutually inverse equivalences.
To this end, we prove that there are natural equivalences $\tau'_F\circ \tau_F\simeq {\rm id}_{\underline{\CM}_F} $ and $\tau_F\circ \tau'_F \simeq {\rm id}_{\overline{\CM}_F}.$ We only demonstrate the proof for the first natural equivalence, as the proof for the second one is similar. Let $X \in \CM$.  Applying the natural equivalences \eqref{eq1} and \eqref{eq2} imply that
\[D\underline{\CM}_F(\tau'_F \tau_F X, X)\simeq F(X, \tau_FX)\simeq D\overline{\CM}_F(\tau_FX, \tau_FX)\]
By applying the equivalence $\tau_F$ we get
\[\underline{\CM}_F(\tau'_F \tau_FX, X)\simeq \underline{\CM}_F(X, X)\]
Define the natural isomorphism $\tau''_F \tau'_FX\lrt X$ as the image of the identity ${\rm id}_{X}$ under this isomorphism.

It remains to show that we have the natural isomorphisms
\[D\underline{\CM}_F(\tau^{-1}_{ F}Y, X)\simeq F(X, Y)\simeq D\overline{\CM}_F(Y, \tau_{ F}X).\]
Let $X \in \CM$. The equality $\tau_F= P^{{(-)}^{-1}} \circ D\circ \Sigma_F \circ P_{(-)}$ implies that 
\[\overline{\CM}_F( - , \tau_FX) = D\circ \Sigma_F(\underline{\CM}_F( - ,X)).\]
This in view of Lemma \ref{Lemma-Sigma}, implies that 
\begin{equation}\label{eq11}
D\overline{\CM}_F( - , \tau_FX) =  \Sigma_F(\underline{\CM}_F( - ,X)) \simeq F(X, - ).
\end{equation}

On the other hand, let $Y \in \CM$. The equality 
\[\tau'_{F}= \Theta^{-1} \circ {\Sigma^{-1}\circ D} \circ \Theta'\] 
implies that 
\[\Theta\tau'_F(Y, - ) = \Sigma^{-1}(\overline{\CM}_F(Y, - )),\]
that, in view of Lemma \ref{Lemma-Sigma} and the fact that $\tau'_F=\tau^{-1}_F$, implies that 
\begin{equation}\label{eq22}
D\underline{\CM}_F(\tau^{-1}_{ F}Y, - ) = \Sigma^{-1}(\overline{\CM}_F(Y, - )) \simeq F( - ,Y).
\end{equation}

Therefore the proof can be completed by comparing the equations \eqref{eq11} and \eqref{eq22}. 
\end{proof}

\begin{remark}
It is natural to use the notation $\tau_{n, F}$ for the relative $n$-Auslander-Reiten translation, but since $n$ is fixed, following our convention, we omit it.
\end{remark}

Theorem \ref{Theroem-relative-tau}, in view of \cite[Theorem 3.5]{HHZ}, implies the following result.

\begin{corollary}\label{exsitenceAlmosSplit}
Let $F$ be a subfunctor of $\Ext^n_{\Lambda}(-,-)$ with enough projective and enough injective objects. The following statements are true.  
\begin{itemize}
\item[$(1)$] For every indecomposable module $X$ not belonging to $\CP(F)$, there exists an $F$-almost split sequence
\begin{equation}
0 \lrt M^0 \lrt M^1 \lrt \cdots \lrt M^n \lrt X \lrt 0 \notag
\end{equation}
in $\CM$.
\item[$(2)$] For every indecomposable module $Y$ not belonging to $\CI(F)$, there exists an $F$-almost split sequence
\begin{equation}
0 \lrt Y \lrt M^1 \lrt M^2 \lrt \cdots \lrt M^n \lrt M^{n+1} \lrt 0 \notag
\end{equation}
in $\CM$.
\end{itemize}
\end{corollary}

\begin{remark}\label{left adjoint}
Let $F\subseteq G$ be subfunctors of $\Ext^n_{\CM}(-, -)$. Assume that $F$ has enough projectives. It is plain that $\CP(G) \subseteq \CP(F)$. This in view of the equivalence $\underline{\CM}_F=\CM/\CP(F)\simeq \frac{\CM/\CP(G)}{\CP(F)/\CP(G)}$, implies that there is a sequence of functors
\[\CP(F)/\CP(G) \hookrightarrow \underline{\CM}_G \st{\pi}\lrt  \underline{\CM}_F.\]
The quotient functor $\pi$ induces an embedding $\ell: \mmod\underline{\CM}_F \hookrightarrow \mmod\underline{\CM}_G$. Since $F$ has enough projectives, $\CP(F) $ is a contravariantly finite subcategory of $\CM$. This implies immediately that $\CP(F)/\CP(G)$ is a contravariantly finite subcategory in the quotient category $\underline{\CM}_G$. Moreover, it can be seen that there exists a left adjoint $\pi_{!}:\mmod \underline{\CM}_G\lrt \mmod \underline{\CM}_F$ to the embedding $\ell$ which is given by sending $H$ to $\pi_{!}H$, where $\pi_{!}H$ is computed as follows
\[\underline{\CM}_G( - , X)\lrt \underline{\CM}_G( - , Y)\lrt H\lrt 0\]
\[\underline{\CM}_F( - , X)\lrt \underline{\CM}_F( - , Y)\lrt \pi_{!}H\lrt 0.\]
There is a similar treatment for the following sequence of functors
\[ \CI(F)/\CI(G) \hookrightarrow \overline{\CM}_G \st{\pi'} \lrt \overline{\CM}_F.\]
Therefore, there exists the left adjoint functor 
$$\pi'_{!}:\mmod \overline{\CM}_G \lrt \mmod \overline{\CM}_F.$$
\end{remark}

\begin{proposition}
Let $F\subseteq G$ be subfunctors of $\Ext^n_{\CM}(-, -)$. Assume that $F$ and $G$ have enough projectives and enough injectives. Then there is the diagram
\[ \xymatrix{ \underline{\CM}_F\ar[r]^{\tau_F} &  \overline{\mathcal{M}}_F \\
    \underline{\CM}_G	\ar[u]^{\pi} \ar[r]^{\tau_G} & \overline{\CM}_G \ar[u]_{\pi'}}\]
of functors, which is commutative up to natural equivalence of functors.
\end{proposition}

\begin{proof}
By definitions of $\Sigma_F$ and $\Sigma_G$ and using the fact that every $F$-exact sequence is $G$-exact, we can easily verify the commutativity of the diagram
\[\xymatrix{\mmod \underline{\CM}_G \ar[rr]^{D\circ \Sigma_G} && \mmod \overline{\CM}_G \\ \mmod \underline{\CM}_F \ar[rr]^{D\circ \Sigma_F}\ar@{^(->}[u] & & \mmod \overline{\CM}_F\ar@{^(->}[u] }\]
By Remark \ref{left adjoint}, we have the left adjoints $\pi_{!}$ and $\pi'_{!}$ to the embeddings $\mmod\underline{\CM}_F \hookrightarrow \mmod \underline{\CM}_G$ and $\mmod\overline{\CM}_F \hookrightarrow \mmod\overline{\CM}_G$, respectively. According to the equivalences $D\circ \Sigma_F$ and $D\circ \Sigma_G$ and the above commutative diagram, we can deduce that the embedding $\mmod \underline{\CM}_F \hookrightarrow \mmod \underline{\CM}_G$ has another left adjoint, which is $(D\circ \Sigma_F)^{-1}\circ \pi'_{!} \circ (D\circ \Sigma_G)$. Because of the uniqueness of the left adjoints, there is a unique natural isomorphism 
$$(D\circ \Sigma_F)^{-1}\circ \pi'_{!} \circ (D\circ \Sigma_G) \simeq \pi_{!}.$$ 
This leads to  the following diagram which commutes up to the isomorphism of functors
\[\xymatrix{\mmod \underline{\CM}_F \ar[rr]^{D\circ \Sigma_F}&& \mmod \overline{\CM}_F \\ \mmod \underline{\CM}_G \ar[rr]^{D\circ \Sigma_G}\ar[u]_{\pi_{!}}&  & \mmod \overline{\CM}_G.\ar[u]_{\pi'_{!}} }\]
Since every left adjoint functor preserves projective objects and trivially equivalences, the above diagram restricts to the subcategories of projective functors of abelian categories involved in the diagram. Next using the equivalences between the subcategories of projective functors and the corresponding additive categories, we obtain the desired commutative diagram.
\end{proof}

The following is an immediate consequence of the above proposition which establishes a nice connection between the relative $n$-Aulsander-Reiten translation and the classical one.

\begin{corollary}
Assume that subfunctor $F \subseteq \Ext^n_{\CM}( - , - )$  has enough projectives and enough injectives.
Then there is a diagram
 \[ \xymatrix{ \underline{\CM}_F\ar[r]^{\tau_F} &  \overline{\mathcal{M}}_F \\
	\underline{\CM}	\ar[u]^{\pi} \ar[r]^{\tau_n} & \overline{\CM} \ar[u]_{\pi'}}\]
which is commutative, up to isomorphism,
\end{corollary}

\begin{proof}
In the above proposition, set $G=\Ext^n_{\CM}( - , - )$.
\end{proof}

In view of this corollary, one can reformulate the isomorphisms given in Theorem \ref{Theroem-relative-tau} in terms of the classical $n$-Auslander-Reiten translations.

\section{Relative Grothendieck groups}\label{Section 5}
Let $\CC$ be a Krull-Schmidt category. The Grothendieck group of $\CC$ with respect to the trivial split exact structure is denoted by ${\rm K}_0(\CC, 0)$. It is defined to be the free abelian group $\oplus_{[X] \in \ind\CC}\Z[X]$ generated by the set $\ind\CC$ of all isomorphism classes of indecomposable objects in $\CC$. We denote by $[X]$ the element in ${\rm K}_0(\CC, 0)$ corresponding to the object $X$ of $\CC$. 

Assume furthermore that $(\CC, \CE)$ is an exact category. We denote by ${\rm Ex}(\CC)$ the subgroup of ${\rm K}_0(\CC, 0)$ generated by
\[\{  [X]-[Y]+[Z] \,| \  X {\rightarrowtail} Y {\twoheadrightarrow} Z \ \text{\ is a conflation in} \ \CC \}.\]
The quotient group ${\rm K}_0(\CC, \CE):={\rm K}_0(\CC, 0)/{\rm Ex}(\CC)$ is called the Grothendieck group of $\CC$ with respect to the exact structure $\CE$. If $\CC$ is an abelian category, the Grothendieck group of $\CC$ will be denoted by ${\rm K}_0(\CC)$.

\s \label{DN3.1} 
Let $\CM$ be an $n$-cluster tilting subcategory of $\mmod\La$, where $n>0$ is a positive integer. Recently, it is shown \cite[Proposition 3.1]{DN} that there exists an isomorphism  
\[\mathbb{H} : {\rm K}_0(\CM, 0) \lrt {\rm K}_0(\mmod \CM)\]
of the Grothendieck groups that maps $[M]$ to $[(-, M)]$. \\

Our aim in this section is to provide a relative version of this fact. Throughout, let $\CM$ be an $n$-cluster tilting subcategory of $\mmod\La$ and $F$ be an additive subfunctor of $\Ext^n_{\CM}( - , - )$. 

\begin{definition}
The Grothendieck group of $\CM$ with respect to $F$, denoted by ${\rm K}_0(\CM, F)$, is defined to be the quotient group ${\rm K}_0(\CM, 0)/{\rm Ex}^n_F(\CM)$, where ${\rm Ex}^n_F(\CM)$ is the subgroup of ${\rm K}_0(\CM, 0)$ generated by the set
\[\{ \sum^{n+1}_{i=0}(-1)^i[M_i] \ | \ 0 \lrt M^0 \lrt M^1\lrt \cdots \lrt M^{n+1}\lrt 0 \ \text{\ is an $F$-exact sequence}\}.\]
As before, since $n$ is fixed, we omit it and write ${\rm Ex}_F(\CM)$ for ${\rm Ex}^n_F(\CM)$. 
\end{definition}

Note that if $n=1$ and $F=\Ext^1_{\La}( - , - )$, then $\CM=\mmod\La$ and ${\rm K}_0(\mmod\La, \Ext^1_{\La})$ is the classical Grothendieck group of $\mmod\La$. See \cite{R} and \cite{DN} for the case when $n>1$ and $F=\Ext^n_{\La}( - , - )$.  

\s \label{Cdefect} Let
\[\delta \colon 0 \lrt M^0\lrt M^1\lrt \cdots \lrt M^n\lrt M^{n+1} \lrt 0 \]
be an $n$-exact sequence in $\CM$. By \cite[Definition 3.1]{JK}, the contravariant defect of $\delta$, denoted by $\delta^*$, is defined by the exact sequence
\[0 \lrt ( - , M^0) \lrt ( - , M^1)\lrt \cdots \lrt ( - , M^n)\lrt ( - , M^{n+1}) \lrt \delta^* \lrt 0\]
of functors. It follows from the definition that every functor $G$ in $\mmod\underline{\CM}$ is a contravariant defect of an $n$-exact sequence; see \cite[Lemma 3.2]{DN}, where they called such functors effaceable. Similarly, one can see that every functor $G$ in $\mmod\underline{\CM}_F$ is a contravariant defect of an $F$-exact sequence. We use this fact without proof.\\

The following proposition shows that the relative Grothendieck group of an $n$-cluster tilting subcategory can be obtained as a factor of the Grothendieck group of an abelian category. 

\begin{proposition}
There is an isomorphism
\[{\rm K}_0(\CM, F) \simeq \frac{{\rm K}_0(\mmod \CM)}{{\rm K}_{F}(\CM)}\]
of abelian groups, where ${\rm K}_{F}(\CM)$ denote the subgroup of ${\rm K}_0(\mmod\CM)$ generated by all $[\delta^*]$, where $\delta$ ranges over $F$-exact sequences in $\CM$. 
\end{proposition}

\begin{proof}
Let
\[\delta \colon 0 \lrt M^0 \lrt M^1\lrt \cdots \lrt M^n \lrt M^{n+1}\lrt 0\]
be an $F$-exact sequence. It is easy to see that the morphism $\mathbb{H}$ of \ref{DN3.1}, maps element $\sum^{n+1}_{i=0}(-1)^i[M_i]$ to $[\delta^*]$. Hence by restricting the isomorphism $\mathbb{H}$ to ${\rm Ex}_F(\CM)$ we get the commutative diagram
\begin{equation}\label{digarm1M}
\xymatrix{ {\rm K}_0(\CM, F) & \frac{{\rm K}_0(\mmod \CM)} { {\rm K}_{F}(\CM)}\\
{\rm K}_0(\CM,0) \ar[u] \ar[r]^{\mathbb{H} \ \ \ } & {\rm K}_0(\mmod \CM) \ar[u]\\
{\rm Ex}_F(\CM)	\ar@{^(->}[u] \ar[r]^{\mathbb{H}| \ \ \ } & {\rm K}_{F}(\CM) \ar@{^(->}[u]}
\end{equation}
with isomorphisms in the rows. This, in turn, induces the isomorphism
\[{\overline{\mathbb{H}}}: {\rm K}_0(\CM,F) {\lrt} \frac{{\rm K}_0(\mmod \CM)} { {\rm K}_{F}(\CM)}\]
of abelian groups, that completes the proof. 
\end{proof}

\s \label{newstructure} Let $\CX = \CP(F)$ be the subcategory of $\CM$ consisting of all $F$-projective modules, where as in our setup, $F$ is an additive subfunctor of $\Ext^n_{\CM}( - , - )$. Let $F^1_{\CX}$ be the class of all short exact sequences $\xi$ in $\mmod\La$ such that $\Hom_{\La}(\CX,\xi)$ is exact. This forms a subfunctor of $\Ext^1_{\La}( - , - )$. By \cite{ASo}, it is indeed a closed subfunctor of $\Ext^1_{\La}( - , - )$, meaning that $(\mmod\La, F^1_{\CX})$ is an exact category, see also \cite{Bu}. We let ${\rm K}_0(\mmod\La, F^1_{\CX})$ denote the Grothendieck group of $\mmod\La$ with respect to this exact structure. \\

Using this, we can record the following corollary. 

\begin{corollary}
There is an isomorphism 
\[{\rm K}_0(\mmod \La, F^1_{\CX}) \simeq \frac{{\rm K}_0(\mmod\mmod \La)}{{\rm K}_{F}(\mmod\La)}\]
of abelian groups, where ${\rm K}_{F}(\mmod\La)$ denote the subgroup of ${\rm K}_0(\mmod\mmod\La)$ generated by the collection of all functors in $\mmod\mmod\La$ that vanish on $\CP(F).$
\end{corollary}

\begin{proof}
By applying the isomorphism \ref{DN3.1} to the $1$-cluster tilting subcategory $\mmod\La$, we obtain the isomorphism
\[\mathbb{H}':{\rm K}_0(\mmod \La, 0) \lrt {\rm K}_0(\mmod\mmod \La)\]
of groups. It is easy to see that the restriction of $\mathbb{H}'$ to ${\rm Ex}_{F^1_{\CX}}(\mmod\La)$ induces the commutative diagram
\begin{equation}\label{DiagramM12}
\xymatrix{ {\rm K}_0(\mmod \La, F^1_{\CX}) & \frac{{\rm K}_0(\mmod \mmod \La)} {{\rm K}_{F}(\mmod\La)}\\
	\mathrm{K}_0(\mmod \La,0) \ar[u] \ar[r]^{\mathbb{H}'} & {\rm K}_0(\mmod\mmod\La) \ar[u]\\
	{\rm Ex}_{F^1_{\CX}}(\mmod\La) \ar@{^(->}[u] \ar[r]^>>>>>>{\mathbb{H}'|} & {\rm K}_{F}(\mmod\La) \ar@{^(->}[u]}
	\end{equation}
where the rows are isomorphisms. This induces the isomorphism
\[\overline{\mathbb{H}'}: {\rm K}_0(\mmod\La, F^1_{\CX}) \lrt \frac{{\rm K}_0(\mmod \mmod \La)} {{\rm K}_{F}(\mmod\La)}\]
of groups, and hence completes the proof.
\end{proof}

\s \label{combine} Consider the restriction functor $\CR:\mmod\mmod\La \lrt \mmod\CM$ that maps any functor $G \in \mmod\mmod\La$ to the composition of $G$ with the inclusion $\CM \hookrightarrow \mmod\La$, that is, $\CR(G) = G|_{\CM}$. Since the functor $\CR$ is exact, it induces a group homomorphism 
\[ \overline{\CR}:{\rm K}_0(\mmod\mmod\La) \lrt {\rm K}_0(\mmod\CM).\] 
It is also clear that the restriction of $\overline{\CR}$ to ${{\rm K}_{F}(\mmod\La)}$ induces the homomorphism 
\[\overline{\CR}|: {{\rm K}_{F}(\mmod\La)} \lrt {\rm K}_F(\CM)\]
of groups.

Combing this restriction with the diagrams \eqref{digarm1M} and \eqref{DiagramM12} induces the commutative diagram
\begin{equation}\label{DiagramM123}
\xymatrix{ {\rm K}_0(\mmod \La, F^1_{\CX})\ar[r]^{\overline{\mathbb{H'}} \ \ \ \ \ } & \frac{{\rm K}_0(\mmod \mmod \La)} {{\rm K}_{F}(\mmod\La)} \ar[r]^{ \ \ \ \overline{\overline{\CR}}}&\frac{{\rm K}_0(\mmod \CM)}{{\rm K}_{F}(\CM)} &{\rm K}_0(\CM, F)\ar[l]_{ \ \ \ \overline{\mathbb{H}}} &\\
	\mathrm{K}_0(\mmod \La,0) \ar[u] \ar[r]^{\mathbb{H'} \ \ \ \ \ } & {\rm K}_0(\mmod\mmod\La) \ar[r]^{ \ \ \ \overline{\CR}} \ar[u] & {\rm K}_0(\mmod \CM)\ar[u] &\mathrm{K}_0(\CM,0)\ar[l]_{ \ \ \ \mathbb{H}}\ar[u] &\\
	{\rm Ex}_{F^1_{\CX}}(\mmod\La) \ar@{^(->}[u] \ar[r]^>>>>>>{\mathbb{H'}|} & {\rm K}_{F}(\mmod\La) \ar@{^(->}[u]\ar[r]^{ \ \ \ \overline{\CR}|}& {\rm K}_{F}(\CM) \ar@{^(->}[u] &{\rm Ex}_F(\CM)\ar[l]_{ \ \ \ \mathbb{H}|}\ar@{^(->}[u] &}
\end{equation}
	
As a result of the top row of the diagram, we obtain the homomorphism
\[\mathbb{T}:=\overline{\mathbb{H}}^{-1}\circ \overline{\overline{\CR}}\circ \overline{\mathbb{H'}}: {\rm K}_0(\mmod \La, F^1_{\CX}) \lrt {\rm K}_0(\CM, F) \]
of groups. 

\s For the proof of our last result in this section, we need to recall another homomorphism from \cite{R}. There is a trivial group homomorphism
\begin{align*}
\mathbb{I}_0: \mathrm{K}_0(\CM,0) & \lrt \mathrm{K}_0(\mmod \Lambda,0).\\
 [M]& \ \mapsto [M]
\end{align*}
By using the same argument as in the Lemma 3.3 of \cite{R}, one can see that $\mathbb{I}_0$ induces a well-defined homomorphism
\begin{align*}
\mathbb{I}:\mathrm{K}_0(\CM, F) \lrt {\rm K}_0(\mmod\Lambda, F^1_{\CX})
\end{align*}
of groups. 

\begin{theorem}\label{eqofgro}
The homomorphism 
\[\mathbb{I} : \mathrm{K}_0(\CM, F) \lrt {\rm K}_0(\mmod\Lambda, F^1_{\CX})\]
is an isomorphism of abelian groups.
\end{theorem}

\begin{proof}
Let $[M] \in {\rm K}_0(\CM, F)$. It follows from the equalities
\begin{align*}
    \mathbb{T}\circ \mathbb{I}([M])&=\overline{\mathbb{H}}^{-1}\circ \overline{\overline{\CR}}\circ \overline{\mathbb{H'}} \circ \mathbb{I}([M])\\
    &= \overline{\mathbb{H}}^{-1} \circ \overline{\overline{\CR}}\circ \overline{\mathbb{H'}}([M])\\
    &= \overline{\mathbb{H}}^{-1} \circ \overline{\overline{\CR}}([\mmod \La(-, M)])\\
    &= \overline{\mathbb{H}}^{-1}([\mmod \La(-, M)|_{\CM}])\\
    &= \overline{\mathbb{H}}^{-1}([\CM(-, M)|_{\CM}])\\
    &=[M].
\end{align*}
that $\mathbb{T}\circ \mathbb{I}={\rm id}_{\mathrm{K}_0(\CM, F)}$. Therefore the group homomorphism $\mathbb{I}$ is injective. We show that $\mathbb{I}$ is surjective. By \cite[Theorem 2.2.3]{I1}, for every module $X$ there exists an exact sequence
\begin{equation}
\eta \colon 0 \lrt M^1_X \lrt \cdots\lrt M^n_X \lrt X\lrt 0 \notag
\end{equation}
such that $M^i_X$, for $i \in \{1, \ldots, n \}$, lies in $\CM$ and the sequence remains exact after applying the functor $\Hom_{\La}(M, -)$ for every $M \in \CM.$ In particular, the above sequence is exact in the exact category $(\mmod \La, F^1_{\CX})$. Hence we have equality 
\[ [X]=\sum^{n-1}_{i=0}(-1)^i[M^i_X] \]
in ${\rm K}_0(\mmod \La, F^1_{\CX})$. According to this equality, to complete the proof, we need only to show that $[M]$, for every $M \in \CM$, lies in the image of $\mathbb{I}$. But this is clear by the definition of $\mathbb{I}$. Hence the proof is complete.
\end{proof}

\begin{remark}
The above theorem not only provides a relative version of \cite[Theorem 3.11]{DN} but also removes the strong assumption that they need on the finiteness of the length of the effaceable functors, i.e. functors that vanish on projective modules.
\end{remark}

\section{Representation type of $\CM$}\label{section 6}
An artin algebra $\La$ is of finite representation type if, up to isomorphism, there are only finitely many indecomposable objects in $\mmod\La$. Butler \cite{But} and Auslander \cite{Au3} proved that $\La$ is of finite representation type if and only if the relations of the Grothendieck group of $\La$ are generated by almost split sequences; in their notation languages ${\rm Ex}(\mmod\La) = {\rm AR}(\mmod\La)$. See \cite[Theorem 3.7]{E} for a version of this result in the context of exact categories and \cite{DN} for a higher version in the context of $n$-cluster tilting subcategories. In this section, we study this result in the context of relative higher homological algebra. Let us begin with the following two facts.

\s \label{deltastar} Let $\delta$ be the $n$-almost split sequence
\[0 \lrt M^0 \st{f^0}\lrt M^1 \st{f^1}\lrt M^2 \lrt \cdots \lrt M^{n} \st{f^{n}}{\lrt} M^{n+1} \lrt 0 \]
in $\CM$. By \ref{almost split}, there exists the sequence
\begin{align*}
0 \lrt \Hom_\La( - , M^0) \lrt \cdots \lrt \Hom_\La( - , M^{n+1}) \lrt\Hom_\La( - , M^{n+1})/\mathcal{J}_\La( - , M^{n+1}) \lrt 0,
\end{align*}
which is exact on $\CM$. Hence, by \ref{Cdefect}, $\delta^*=\Hom_\La( - , M^{n+1})/\mathcal{J}_\La( - , M^{n+1})$.

When $M^{n+1}$ is an indecomposable module, for any $X\in\ind\CM$,
\[\delta^*(X)=\left\{\begin{array}{lll}
\ \ \ \ \ \ \ 0  &  {\rm if} \ X \not\simeq  M^{n+1};\\
\\
\frac{{\rm End}_{\La}(M^{n+1})}{\mathcal{J}_\La(M^{n+1}, M^{n+1})}   & {\rm if} \ X\simeq  M^{n+1}.
\end{array}\right.\]

\s \label{KV3.4} A functor $F$ in $\mmod \CM$ is said to be of finite length if it has a finite composition series. The support of $F$, denoted by {\rm supp}$(F)$, is a set of representatives of isomorphism classes of all indecomposable objects $M \in \CM$ such that $F(M) \not = 0$. We know from \cite[Lemma 3.4]{KV} that a functor $F$ has finite length if and only if the support of $F$ is finite.

\begin{proposition}\label{Groupisomorphism}
Let $F$ be an additive subfunctor of $\Ext^n_{\CM}( - , - )$. Let ${\rm AR}_F(\CM)$ denote the subgroup of ${\rm Ex}_F(\CM)$ generated by the set
\[\{\sum^{n+1}_{i=0}(-1)^i[M_i] \,| \ 0 \lrt M^0 \lrt M^1\lrt \cdots \lrt M^{n+1}\lrt 0 \ \text{\ is an $F$-almost split  sequence}\}.\]
The following statements are equivalent.
\begin{itemize}
\item[(1)] ${\rm Ex}_F(\CM)={\rm AR}_F(\CM)$.
\item[(2)] Every functor $G \in \mmod\underline{\CM}_F$ has finite length.
\end{itemize}
\end{proposition}

\begin{proof}
$(1) \Rightarrow (2)$. Let $G$ belong to $\mmod\underline{\CM}_F$. Then there is an $F$-exact sequence
\[\delta \colon 0 \lrt M^0\lrt M^1\lrt \cdots \lrt M^n\lrt M^{n+1}\lrt 0,\]
such that $\delta^*=G.$ Since ${\rm Ex}_F(\CM)={\rm AR}_F(\CM)$, there exist integers $\lambda_1, \ldots, \lambda_m$ and $F$-almost split sequences  
$$\delta_j \colon 0 \lrt X^0_j\lrt X^1_j \lrt \cdots \lrt X^{n}_j \lrt X^{n+1}_j \lrt 0,$$
for $j \in \{1, \ldots, m\}$, such that  
\[\sum^{n+1}_{i=0}(-1)^i[M^i]=\sum^{m}_{j=1}\lambda_j(\sum^{n+1}_{i=0}(-1)^i[X^i_j]) \]
in ${\rm K}_0(\CM, 0)$.

By applying the group isomorphism $\mathbb{H}$, given in \ref{DN3.1}, we get the equality
\begin{equation}\label{Equation1}
\mathbb{H}(\sum^{n+1}_{i=0}(-1)^i[M^i])=\mathbb{H}(\sum^{m}_{j=1}\lambda_j(\sum^{n+1}_{i=0}(-1)^i[X^i_j]))
\end{equation}
in ${\rm K}_0(\mmod \CM)$.
Hence
$$[G]=\sum^m_{j=1}\lambda_j[\delta^*_j].$$
Lemma 3.6 of \cite{KV} implies that
$${\rm supp}(G) \subseteq \bigcup^m_{j=1}{\rm supp}(\delta^*_j) $$
Since $\delta_j$'s are corresponded to simple functors in $\mmod \CM$, $|{\rm supp}(\delta^*_j)|=1$. Hence $G$ has finite support and therefore by \ref{KV3.4} is of finite length.

$(2)\Rightarrow (1)$. This follows by applying a similar argument as in the proof of Theorem 3.11 of \cite{DN}. For the convenience of the reader, we recall the sketch of the proof. Let 
\[\delta \colon 0 \lrt M^0\lrt M^1\lrt \cdots \lrt M^n\lrt M^{n+1}\lrt 0\]
be an $F$-exact sequence in ${\rm Ex}_F(\CM)$. By assumption, $\delta^*$ has finite length. So
\[[\delta^*] = [S_1] + [S_2] + \cdots + [S_t],\]
where $S_1, \ldots, S_t$ are simple composition factors of $\delta^*$. By Proposition 2.3 of \cite{Au2}, $S_i$, for each $i$, can be written as $\Hom_\La( - , X^{i})/\mathcal{J}_\La( - , X^{i})$, where $X^i \in \CM$ is an indecomposable module. Moreover, in view of \ref{deltastar}, for every indecomposable module $Y \in \CM$ which is not isomorphic to $X^i$, we have $S_i(Y)=0.$ So for all $i \in \{1, \ldots, t\}$, $\delta^*(X^i) \not = 0.$ This, in particular, implies that $X^i$, for all $i$, is not $F$-projective, as otherwise $\delta^*(M^i) = 0.$ So, by Corollary \ref{exsitenceAlmosSplit}, for each $X^i$ there exists an $F$-almost split sequence
\[\delta_i \colon 0 \lrt N^0_i \lrt N^1_i\lrt \cdots \lrt N^n_i \lrt X^i \lrt 0.\]
This follows that $[\delta^*_i] \in {\rm AR}_F(\CM)$ and hence completes the proof.
\end{proof}

In order to prove the main theorem of this section, i.e. Theorem \ref{main2}, we need some preparation, which is done in the following steps.

\s Recall \cite[Theorem 2.2.3]{I1} that for every $X\in\mmod\La$, there exists an exact sequence
\begin{equation}
\eta \colon 0 \lrt M^1_X \lrt \cdots\lrt M^n_X \lrt X\lrt 0. \notag
\end{equation}
such that $M^i_X \in \CM$, for $i \in \{1, \ldots, n \}$, and the sequence remains exact after applying the functor $\Hom_{\La}(\CM, -)$. 
The index with respect to $\CM$ is defined in \cite[Definition 1.4]{R} by the map
\begin{align*}
\mathrm{Ind}_{\CM} : \ & \mmod\La \lrt \mathrm{K}_0(\CM,0)\\
& \ \ \ X \mapsto \sum^{n}_{i=1}(-1)^i[M^i_X]
\end{align*}
Note that by \cite[ Remark 2.1]{R}, $\mathrm{Ind}_{\CM}$ is well-defined.

By \cite[p. 65]{DN}, an element $\beta_M$ in $\mathrm{K}_0(\CM,0)$ is associated to a module $M\in{\rm ind}\mbox{-}\CM$ as follows. If $M$ is projective, we set $$\beta_M:=[M]-{\Ind}_\CM(\rad(M)),$$ where $[M]$ is considered in ${\rm K}_0(\CM,0)$. If $M$ is non-projective, there exists an $n$-almost split sequence
\[0 \lrt M^0 \lrt M^1 \lrt \cdots \lrt M^n \lrt M \lrt 0\]
in $\CM$ and we set $$\beta_M:= \displaystyle\sum_{i=0}^{n+1}(-1)^i[M^i],$$ where $M^{n+1}=M$.
Since by \cite[Sec. 3.1.1]{I1}, $n$-almost split sequences are uniquely determined by each of their right or left terms, $\beta_M$ is well-defined. 

\s \label{au1} By \cite[Lemma 3.7]{DN}, for every $X, M \in {\rm ind}\mbox{-}\CM$, we have
\[
\langle[X],\beta_M\rangle:=\left\{
\begin{array}{lll}
0 & {\rm if} \ X \not\simeq M\\
\\
l_M & {\rm if } \ X \simeq M,
\end{array}\right.\]
where $\langle -, -\rangle:\mathrm{K}_0(\mmod \La,0)\times {\rm K}_0(\mmod\La,0) \lrt \mathbb{Z}$
is the bilinear form such that for modules $X$ and $Y$ in $\mmod\La$, $\langle X, Y\rangle$ equals the length of $\mathrm{Hom}_\La(X, Y)$ as $R$-module and $l_M$, for $M \in {\rm ind}\mbox{-}\CM$, denotes the length of ${\rm End}_{\La}(M)/\mathcal{J}_\La(M,M)$ as $R$-module.

\s \label{AHVTheorem} Let $\CX$ be a subcategory of $\mmod\La$. Similar to \ref{newstructure}, by $(\mmod\La, {F}^1_{\CX})$ we mean the exact structure on $\mmod\La$ induced by all short exact sequences in $\mmod \La$ that remain exact with respect to the functor $\Hom_{\La}(\CX, - )$, the so-called $\CX$-proper exact sequences; compare \ref{notation2.10}. 

By \cite[Theroem 3.3]{AHV}), the assignment $X \mapsto \CX( - , X)$ induces the equivalence
\[ \mathbb{D}^{\rm b}(\mmod \La, F^1_{\CX}) \simeq \mathbb{D}^{\rm b}(\mmod \CX)\]
of triangulated categories, provided $\CX$ is a contravariantly finite subcategory of $\mmod\La$. 

\s Let $\CT$ be an essentially small Krull-Shmidt triangulated category with suspension $\Sigma$. Denote by ${\rm G}(\CT)$ the free abelian group generated by the isomorphism classes of indecomposable objects and by ${\rm G}_0(\CT)$ the subgroup of ${\rm G}(\CT)$ generated by
\[\{ [X]-[Y]+[Z] \ | \ X\lrt Y\lrt Z\lrt \Sigma X \ {\rm{is \ an \ exact \ triangle \ in}} \ \CT\}.\]
The Grothendieck group ${\rm K}_0(\CT)$ of $\CT$ is defined as the quotient ${\rm G}(\CT)/{\rm G}_0(\CT)$ of abelian groups.

\s \label{LemmaexactGro} Let $\CA$ be an exact category. By \cite[ Lemma 4.1.12]{Ke}, the embedding $\CA \lrt \mathbb{D}^{\rm b}(\CA)$ defined by $X \mapsto \overline{X}$, where the complex $\overline{X}$ is concentrated in degree zero with $X$ in the zeroth degree, yields a group homomorphism
\[\eta_{\CA}:{\rm K}_0(\CA)\lrt {\rm K}_0(\mathbb{D}^{\rm b}(\CA)).\]

By \cite[Lemma 4.1.8]{Ke}, the assignment $X=(X^i, d^i)_{i \in \mathbb{Z}}\mapsto \sum_{i\in \mathbb{Z}}(-1)^i[X^i]$ induces the isomorphism
\[\xi_{\CA}:{\rm K}_0(\mathbb{D}^{\rm b}(\CA))\lrt {\rm K}_0(\CA),\]
with $\eta_{\CA}$ as its inverse.

\begin{lemma}\label{Delta}
Let $F$ be a subfunctor of $\Ext^n_{\CM}( - , - )$ with the property that $\CP(F)$ is a contravariantly  finite subcategory of $\CM$. There exists an isomorphism
\[\Delta:{\rm K}_0(\mmod \CP(F)) \stackrel{\sim}\lrt {\rm K}_0(\mmod \La, F^1_{\CP(F)})\]
of abelian groups.
\end{lemma}

\begin{proof}
Since $\CP(F)$ is a contravariantly finite subcategory of $\CM$, $\mmod\CP(F)$ is an abelian category. Apply \ref{LemmaexactGro} to $\mmod\CP(F)$, we get the isomorphism
\[\eta_{\mmod \CP(F)}:{\rm K}_0(\mmod \CP(F))\lrt {\rm K}_0(\mathbb{D}^{\rm b}(\mmod \CP(F))).\]
Now apply \ref{AHVTheorem} to get the isomorphism
\[{\rm K}_0(\mathbb{D}^{\rm b}(\mmod \CP(F))) \lrt {\rm K}_0(\mathbb{D}^{\rm b}(\mmod \La, F^1_{\CP(F)})).\]
Another application of \ref{LemmaexactGro}, this time to the exact category $(\mmod \La, F^1_{\CP(F)})$ gives the isomorphism
\[\xi_{\mmod\La}: {\rm K}_0(\mathbb{D}^{\rm b}(\mmod \La, F^1_{\CP(F)})) \lrt {\rm K}_0(\mmod \La, F^1_{\CP(F)}).\]
Their composition is the desired $\Delta$.
\end{proof}

We also need the following remark in the proof of the main theorem.

\begin{remark}\label{Proof of Claim}
Let $F$ be a subfunctor of $\Ext^n_{\CM}( - , - )$ such that $\CP(F)$ is of finite type. This aasmption implies that $\mmod\CP(F)$ is equivalent to $\mmod\Gamma$, where $\Gamma$ is the endomorphism algebra of an additive generator of $\CP(F).$ 
Apply \ref{au1} to the $1$-cluster tilting subcategory $\mmod \Gamma$. Then, under the equivalence between $\mmod\Gamma$ and $\mmod\CP(F)$, the $\beta$'s for $\mmod\Gamma$ are corresponded to $\beta_{\CP(F)(-, X)}$ for $\mmod \CP(F)$, where $X$ is an indecomposable module in $\CP(F)$. Accordingly, we have the bilinear form $$\langle -, -\rangle:\mathrm{K}_0(\mmod \CP(F),0)\times\mathrm{K}_0(\mmod\CP(F),0)\lrt \mathbb{Z}$$ with the condition
\[\langle[V],\beta_{\CP(F)(-, X)}\rangle =\left\{
\begin{array}{ll}
0 \quad \ \ {\rm if} \ \ V  \ {\not \simeq} \ \CP(F)(-, X),\\
\\
l_X  \quad {\rm if} \ \ V \simeq  \CP(F)(-, X).
\end{array}\right. \]
Note that thanks to the isomorphism 
\[\End_{\La}(X)/\mathcal{J}(X, X)\simeq \End(\CP(F)(-, X))/\mathcal{J}(\CP(F)(-, X), \mathcal{J}(-, X)))\] 
the length $l_X$ of $\End_{\La}(X)/\mathcal{J}(X, X) $ as $R$-module is equal to the length of  
\[\End(\CP(F)(-, X))/\mathcal{J}(\CP(F)(-, X), \mathcal{P}_n(-, X))).\] 
\end{remark}

Now we are ready to prove the main theorem of this section. 

\begin{theorem}\label{main2}
Let $F$ be a subfunctor of $\Ext^n_{\CM}( - , - )$ such that $\CP(F)$ is of finite type. Then the following statements are equivalent.
\begin{itemize}
\item[$(1)$] ${\rm Ex}_F(\CM)={\rm AR}_F(\CM)$.
\item[$(2)$] Every functor $G \in \mmod\underline{\CM}_F $ has finite length.
\item[$(3)$] $\CM$ is of finite type.
\end{itemize}
\end{theorem}

\begin{proof}
The fact that $(1)$ and $(2)$ are equivalent follows from Proposition \ref{Groupisomorphism}. The statement $(2)$ trivially follows from $(3)$. In fact, if $\CM$ is of finite type, then $\mmod\CM_F$ is the category of modules over an artin algebra. Now assume that $(2)$ is true. Since $\CP(F)$ is of finite type, we can consider $\mmod\CP(F)$ as the module category over the endomorphism algebra of an additive generator of $\CP(F).$

By composing the group isomorphism $\mathbb{T}:{\rm K}_0(\mmod \La, F^1_{\CX})\lrt {\rm K}_0(\CM, F)$ of \ref{combine} and Theorem \ref{eqofgro} with the group isomorphism $\Delta$ of Lemma \ref{Delta} we obtain the isomorphism
\[\mathbb{T}\circ \Delta:{\rm K}_0(\mmod \CP(F))\lrt {\rm K}_0(\CM, F)\]
of groups. 

In view of \cite[Theorem I.1.7]{ARS}, $\{[\CP(F)( - , X)]-[{\rm rad}( - , X)] \,|\, X \in {\rm ind}\mbox{-}\CP(F)\}$ is a free basis for $\mathrm{K}_0(\mmod \CP(F))$.
For any $X \in {\rm ind}\mbox{-}\CP(F)$, we set $\sigma_X:=\mathbb{T} \circ \Delta([\CP(F)( - , X)]-[{\rm rad}( - , X)])$. Since $\mathbb{T} \circ \Delta$ is an isomorphism, the set
\[\{\sigma_X \mid  X \in {\rm ind}\mbox{-}\CP(F) \}\]
is a free basis for ${\rm K}_0(\CM, F)$. Moreover, the set
\[\{\beta_A \mid A \in {\rm ind}\mbox{-}\CM \setminus {\rm ind}\mbox{-}\CP(F)  \}\]
forms a free basis for ${\rm Ex}_F(\CM)$. Hence we get the free basis
\[\{\sigma_X\mid X \in {\rm ind}\mbox{-}\CP(F) \} \cup \{\beta_A \mid A \in {\rm ind}\mbox{-}\CM \setminus {\rm ind}\mbox{-}\CP(F)  \}. \]
for ${\rm K}_0(\CM, 0)$. Since $D\La\in\CM$, we can write
\[ [D\La]=\displaystyle\sum_{A\in {\rm ind}\mbox{-}\CM \setminus {\rm ind}\mbox{-}\CP(F) }\lambda_A\beta_A+\displaystyle\sum_{X \in {\rm ind}\mbox{-}\CP(F) }\lambda_X\sigma_X,\]
where $\lambda_A\neq 0$ for only finitely many $A\in {\rm ind}\mbox{-}\CM$. Since every $Y\in{\rm ind} \mbox{-}\CM $ has injective envelope in $\CM$, we have
\[\langle[Y],[D\La]\rangle=\displaystyle\sum_{A\in {\rm ind}\mbox{-}\CM \setminus {\rm ind}\mbox{-}\CP(F)}\lambda_A\langle[Y],\beta_A\rangle+\displaystyle\sum_{X\in {\rm ind}\mbox{-}\CP(F)}\lambda_X\langle[Y],\sigma_X\rangle \neq0.\]

By \ref{au1},
\[\displaystyle\sum_{A\in {\rm ind}\mbox{-}\CM \setminus {\rm ind}\mbox{-}\CP(F)}\lambda_A\langle[Y],\beta_A\rangle=\lambda_Yl_Y,\] for every $X\in{ \rm ind}\mbox{-}\CM$. 

We claim that
\begin{equation}\label{claim}
\displaystyle\sum_{X\in  {\rm ind}\mbox{-}\CP(F)}\lambda_X\langle[Z],\sigma_X\rangle=0,
\end{equation}
for every $Z\in {\rm ind}\mbox{-}\CM \setminus {\rm ind}\mbox{-}\CP(F)$. To see this apply the group isomorphism $\Delta^{-1}\circ \mathbb{T}^{-1}$ to the left-hand side of this equation to get the equalities
\begin{align*}
   \displaystyle\sum_{X\in  {\rm ind}\mbox{-}\CP(F)}\lambda_X\langle[Z],\sigma_X\rangle)&=\displaystyle\sum_{X\in  {\rm ind}\mbox{-}\CP(F)}\lambda_X\langle\Delta^{-1}\circ \mathbb{T}^{-1}([Z]),\Delta^{-1}\circ \mathbb{T}^{-1}(\sigma_X)\rangle\\
   &=\displaystyle\sum_{X\in  {\rm ind}\mbox{-}\CP(F)}\lambda_X\langle [\CM(-, Z)|_{\CP(F)}],\beta_{\CP(F)(-, X)}\rangle.
\end{align*}
But, in view of the Remark \ref{Proof of Claim}, for $X \in {\rm ind}\mbox{-}\CP(F)$, $\la_X \neq 0$ if and only if $\CM(-, Z)|_{\CP(F)}\simeq \CP(F)(-, X)$. Since $\La$ belongs to $\CP(F)$, this implies that $Z\simeq X$, which is a contradiction. So the claim is proved.

This, in particular, implies that every $Y \in \ind\CM$ is isomorphic to either an indecomposable module $A \in {\rm ind}\mbox{-}\CM \setminus {\rm ind}\mbox{-}\CP(F)$ with $\la_A \neq 0$ or an indecomposable module in ${\rm ind}\mbox{-}\CP(F)$. Therefore, since $\CP(F)$ is of finite type, $\CM$ is also of finite type. This completes the proof.
\end{proof}

Our last corollary reproves Theorem 3.9 of \cite{DN}.

\begin{corollary}
Let $\CM$ be an $n$-cluster tilting subcategory of $\mmod \La$. The following statements are equivalent.
\begin{itemize}
\item[$(1)$] ${\rm Ex}(\CM)={\rm AR}(\CM)$,i.e. the $n$-almost split sequences form a basis for the relations for the Grothendieck group of $\CM$.
\item[$(2)$] Every functor $G \in \mmod\underline{\CM}$ has finite length.
\item[$(3)$] $\CM$ is of finite type.
\end{itemize}
\end{corollary}

\begin{proof}
In the above theorem set $F=\Ext^n_{\La}( - , -)$ and note that $\CP(F)=\CP(\La)$ is of finite type.    
\end{proof}

\section*{Acknowledgments}
This research is supported by the National Natural Science Foundation of China (Grant No.\@ 12101316). The work of the second author is based on research funded by Iran National Science Foundation (INSF) under project No. 4001480. The research of the third author is partially supported by the Belt and Road Innovative Talents Exchange Foreign Experts project (Grant No. \@ DL2023014002L).

\end{document}